\documentclass[preprint]{imsart}

\RequirePackage[colorlinks]{hyperref}
\usepackage[dvipdfmx]{graphicx}
\usepackage{amsthm,amsmath,amssymb}
\usepackage{natbib}
\usepackage[dvips]{geometry}
\startlocaldefs
\newtheorem{theo}{Theorem}
\newtheorem{lemm}{Lemma}
\newcommand{\argmin}{\mathop{\rm arg~min}\limits}
\newcommand{\Tr}{\mathop{\rm tr}\nolimits}
\newcommand{\ordz}{O(1)}
\newcommand{\ordi}{O(n^{-1/2})}
\newcommand{\ordii}{O(n^{-1})}
\newcommand{\ordiii}{O(n^{-3/2})}
\newcommand{\ordiiii}{O(n^{-2})}
\endlocaldefs

\begin{document}

\begin{frontmatter}

\title{Higher-order accuracy of multiscale-double bootstrap for testing regions}
\runtitle{multiscale-double bootstrap}
\author{\fnms{Hidetoshi}
 \snm{Shimodaira}\corref{}\ead[label=e1]{shimo@sigmath.es.osaka-u.ac.jp}\thanksref{t1}}
\thankstext{t1}{
Supported in part by Grant
KAKENHI (20500254, 24300106) from MEXT of Japan.}
\address{
Division of Mathematical Science\\Graduate School of
 Engineering Science\\
Osaka University\\
 1-3 Machikaneyama-cho\\Toyonaka, Osaka, Japan\\
 \printead{e1}}

\runauthor{H.~SHIMODAIRA}

\begin{abstract}
We consider hypothesis testing for the null hypothesis being represented
as an arbitrary-shaped region in the parameter space.  We compute an
approximate $p$-value by counting how many times the null hypothesis
holds in bootstrap replicates.  This frequency, known as bootstrap
probability, is widely used in evolutionary biology, but often reported
as biased in the literature.  Based on the asymptotic theory of
bootstrap confidence intervals, there have been some new attempts for
adjusting the bias via bootstrap probability without direct access to
the parameter value. One such an attempt is the double bootstrap which
adjusts the bias by bootstrapping the bootstrap probability.  Another
new attempt is the multiscale bootstrap which is similar to the
$m$-out-of-$n$ bootstrap but very unusually extrapolating the bootstrap
probability to $m=-n$.  In this paper, we employ these two attempts at
the same time, and call the new procedure as multiscale-double
bootstrap. By focusing on the multivariate normal model, we investigate
higher-order asymptotics up to fourth-order accuracy.  Geometry of the
region plays important roles in the asymptotic theory.  It was known in
the literature that the curvature of the boundary surface of the region
determines the bias of bootstrap probability.  We found out that the
``curvature of curvature'' determines the remaining bias of double
bootstrap.  The multiscale bootstrap removes these biases.  The
multiscale-double bootstrap is fourth order accurate with coverage
probability erring only $\ordiiii$, and it is robust against
computational error of parameter estimation used for generating
bootstrap replicates from the null distribution.
\end{abstract}



\end{frontmatter}

\section{Introduction} \label{sec:intro}

We would like to compute approximate $p$-values by bootstrap methods for
testing null hypothesis $H_0: \mu\in H$ against alternative $H_1:
\mu\not\in H$ for a $q+1$ ($\ge2$) dimensional unknown parameter vector
$\mu\in \mathbb{R}^{q+1}$ and an arbitrary-shaped region $H \subset
\mathbb{R}^{q+1}$.  This is the problem of regions discussed in
\cite{Efron:Halloran:Holmes:1996:BCL} and
\cite{Efron:Tibshirani:1998:PR}, where the geometry of the shape of $H$
plays important roles. Their geometric argument is based on the
bias-corrected (BC) bootstrap confidence interval of
\cite{Efron:1985:BCI} for the multivariate normal model
\begin{equation} \label{eq:ynorm}
 Y \sim N_{q+1}(\mu, I_{q+1})
\end{equation}
with mean $\mu$ and covariance identity matrix $I_{q+1}$.  Similar
geometric argument is found in \cite{Efron:1987:BBC},
\cite{Diciccio:Efron:1992:MAC}, and \cite{Shimodaira:2004:AUT} for
exponential family of distributions up to terms of $\ordii$.  We focus
on the multivariate normal model (\ref{eq:ynorm}) in this paper, and
investigate higher-order asymptotics up to terms of $\ordiii$ for
fourth-order accuracy, hoping to get insights into more general
situations.

A simple example is the case of spherical region in
\cite{Efron:Tibshirani:1998:PR}.  Consider $n$ independent random
variables $X_1,\ldots, X_n \sim N_{q+1}(\eta,I_{p+1})$, and the null
hypothesis $\|\eta\|\le 1$, where $\|\eta\|^2=\eta_1^2+\cdots +
\eta_{p+1}^2$. The problem is also described in a transformed variable
$Y=\sqrt{n}\bar{X}$ with mean $\mu=\sqrt{n}\eta$ so that the region is
$H=\{\mu : \|\mu\| \le \sqrt{n}\,\}$. The dependency on $n$ is implicit
in our notation.  This example is simple enough to compute the exact
$p$-value as $P(\|Y\|^2 \ge \|y\|^2)$ by knowing that $\|Y\|^2$ follows
$\chi^2_{p+1}$, the chi-square distribution with degrees of freedom
$p+1$, of non-centrality $\|\mu\|^2$. However, it is not so easy to
compute the exact $p$-value for an arbitrary-shaped region $H$.

Having an observation $y\in\mathbb{R}^{q+1}$ of $Y$, we may generate
many replicates of $Y$ by the parametric bootstrap
\begin{equation} \label{eq:yboot}
 Y^*  \sim N_{q+1}(y, \sigma^2 I_{q+1})
\end{equation}
for some $\sigma^2>0$.  This corresponds to the non-parametric
``$m$-out-of-$n$'' bootstrap of \cite{Bickel:Gotze:vanZwet:1997:RFT} and
\cite{Politis:Romano:1994:LSC} with $\sigma^2=n/m$.  For the spherical
example, we may compute $Y^* = \sqrt{n}(X_1^*+\cdots + X_m^*)/m$ by
resampling $\{X_1^*,\ldots,X_m^*\}$ with replacement from
$\{x_1,\ldots,x_n\}$.  In this paper, we do not pursue the
non-parametric bootstrap, but focus on (\ref{eq:yboot}) for extending
the asymptotic theory of \cite{Efron:1985:BCI}.

Generating many $Y^*$'s, we count how many times they fall in $H$.  This
frequency is called as \emph{bootstrap probability} (BP) and it has been
used extensively since \cite{Felsenstein:1985:CLP} for approximating the
$p$-value of testing phylogenetic trees in evolutionary biology.  It is
also named ``empirical strength probability'' in
\cite{Liu:Singh:1997:NLP}. Although the BP works as an approximate
$p$-value in the frequentist sense, it is often reported as biased and
there have been some attempts for improving the accuracy;
\cite{Hillis:Bull:1993:ETB}, \cite{Felsenstein:Kishino:1993:ITS},
\cite{Newton:1996:BPL}, \cite{Efron:Halloran:Holmes:1996:BCL},
\cite{Efron:Tibshirani:1998:PR},
\cite{Shimodaira:2002:AUT,Shimodaira:2004:AUT,Shimodaira:2008:TRN}.

Assuming sufficiently large number of replicates, we define the BP as
\[
 \mathrm{BP}_{\sigma^2}(H|y) = P_{\sigma^2}(Y^* \in H | y),
\]
where $P_{\sigma^2}(\cdot|y)$ indicates the probability with respect to
($\ref{eq:yboot}$). The variance is usually $\sigma^2=1$ and we simply
denote BP or $ \mathrm{BP}(H|y)$ for $\mathrm{BP}_{1}(H|y)$.  BP is
interpreted as the Bayesian posterior probability of $H$ under
(\ref{eq:ynorm}), because the posterior distribution is $\mu|y \sim
N_{p+1}(y,I_{p+1})$ for the improper uniform prior distribution.

For a specified significance level $0<\alpha<1$, we will reject $H_0$ if
$\mathrm{BP}<\alpha$. 
It follows from eq.~(2.22) of \cite{Efron:Tibshirani:1998:PR}
that the rejection probability is expressed as
\begin{equation} \label{eq:rejprob-bp}
 P\Bigl( \mathrm{BP}(H|Y) < \alpha \Bigr) = \Phi(z_\alpha +
 2\gamma_1) + \ordii
\end{equation}
for $\mu\in\partial H$, where $\gamma_1=\ordi$ is the \emph{mean
curvature} of $\partial H$ at $\mu$ in terms of differential geometry.
Here $\partial H$ denotes the boundary surface of the region $H$,
$\Phi(\cdot)$ is the cumulative distribution function of $N(0,1)$, and
$z_\alpha=\Phi^{-1}(\alpha)$.  A generalization of (\ref{eq:rejprob-bp})
will be proved later in Theorem~\ref{thm:bpacc}.  The rejection
probability of unbiased tests should be equal to $\alpha$ for
$\mu\in\partial H$, and the bias is defined as the deviation of
rejection probability from $\alpha$.  According to
(\ref{eq:rejprob-bp}), the bias of BP is determined mostly by the mean
curvature, which is zero, say, if $\partial H$ is flat. More generally,
the mean curvature is zero everywhere on a ``minimal surface'' that
locally minimizes its area like soap membranes.  We may reject $H_0$ too
much (large type-I error and many false positives) if the curvature is
positive, and reject $H_0$ too little (conservative and few true
discoveries) if the curvature is negative. The sign of $\gamma_1$ is
defined in the way that $\gamma_1>0$ when $\partial H$ is curved toward
$H$.

The bootstrap iteration is a general idea applicable to a wide range of
problems for improving accuracy, and it has been applied to
bootstrap confidence intervals of a real parameter;
\cite{Hall:1986:BCI}, \cite{Beran:1987:PRL}, \cite{Loh:1987:CCC},
\cite{Hinkley:Shi:1989:ISN}, \cite{Martin:1990:BIC},
\cite{Hall:1992:BEE}, \cite{Efron:Tibshirani:1993:ITB},
\cite{Newton:Geyer:1994:BRM}, \cite{Lee:Young:1995:AIB},
\cite{Diciccio:Efron:1996:BCI}, \cite{Hall:Maesono:2000:WBA}.  From the
duality of confidence intervals and hypothesis testing, we may compute a
$p$-value from the iterated bootstrap confidence intervals of a real
parameter, say, $\|\mu\|$ for the spherical example.  However,
additional consideration is needed for computing the $p$-value only from
the frequency of $\{ \|y^*\| \le \sqrt{n} \,\}$ without access to the
bootstrap distribution of $\|y^*\|$.  \cite{Efron:Tibshirani:1998:PR}
applied the bootstrap iteration to BP for adjusting the bias, and called
the bias-corrected BP as a calibrated confidence level.  In this paper,
we call it as \emph{double bootstrap probability} (DBP).

Similar to the bias of BP, the remaining bias of DBP is again
interpreted as a geometric quantity of $\partial H$.  Let
$\beta_3=\ordiii$ be the ``mean curvature of the mean curvature'' of
$\partial H$. We found that $\beta_3$ determines
the bias of DBP. In fact, the rejection probability is 
\begin{equation} \label{eq:rejprob-dbp}
 P\Bigl( \mathrm{DBP}(H|Y) < \alpha\Bigr) = \Phi(z_\alpha -
 2\beta_3) + \ordiiii
\end{equation}
as shown in Theorem~\ref{thm:dbpacc}. Related results are given in
\cite{Hall:1992:BEE} and \cite{Lee:Young:1995:AIB} for the coverage
probability of the iterated bootstrap confidence intervals under the
smooth function model. We can tell from (\ref{eq:rejprob-dbp}) that DBP
is very accurate for the spherical example, because $\beta_3=0$ for
spheres. For constant-mean-curvature surfaces, such as plane, cylinder,
sphere, or intuitively soap bubbles, we have always $\beta_3=0$, and DBP
is very accurate.  For other surfaces, however, the magnitude of
$\beta_3$ can be large.

In this paper, we discuss several bootstrap methods for improving the
accuracy of BP. An approximately unbiased $p$-value is said to be $k$-th
order accurate if the bias is $O(n^{-k/2})$ asymptotically. BP is only
first order accurate, and DBP is third order accurate. We attempt
improving BP and DBP via the \emph{multiscale bootstrap} of
\cite{Shimodaira:2002:AUT,Shimodaira:2004:AUT,Shimodaira:2008:TRN}. A
key idea is to change $\sigma^2$ in (\ref{eq:yboot}).  We derive the
scaling-law of BP and DBP with respect to $\sigma^2$, and extrapolate
these values formally to $\sigma^2=-1$, or $m=-n$ in the non-parametric
bootstrap. The idea is analogous to the SIMEX, simulation-extrapolation,
method for measurement error models of \cite{Cook:Stefanski:1994:SIE}.
It turns out that $\gamma_1$ in (\ref{eq:rejprob-bp}) and $\beta_3$ in
(\ref{eq:rejprob-dbp}) disappear as $\sigma^2$ approaching $-1$.  Thus
the multiscale bootstrap improves both BP and DBP; the bias-corrected BP
is third-order accurate, and the bias-corrected DBP is fourth-order
accurate. This is the main thrust of the paper. We will prove the main
results in Section~\ref{sec:asymptotic-analysis} after preparing
geometric tools in Section~\ref{sec:geometry}.

The bias-corrected BP via multiscale bootstrap has been already used for
testing phylogenetic trees in \cite{Shimodaira:Hasegawa:2001:CAC} and
hierarchical clustering in \cite{Suzuki:Shimodaira:2006:PRA}, and the
hypothesis test is referred to as ``approximately unbiased'' (AU) test
in the literature. For the newly proposed bias-corrected DBP, we call the
procedure as \emph{multiscale-double bootstrap}, and the hypothesis test
as ``double approximately unbiased'' (DAU) test. This procedure is new
and different from the two-step multiscale bootstrap of
\cite{Shimodaira:2004:AUT} which adjusts AU \emph{without}
double-bootstrapping for exponential family of distributions.

\section{Conventional testing procedures} \label{sec:lrtest}

For representing $H$, we use $(u,v)$ coordinates with
$u=(u_1,\ldots,u_q)\in \mathbb{R}^q$ and $v\in \mathbb{R}$.  Given a
smooth function $h(u)$ of $u\in \mathbb{R}^q$, we specify a region as $
\mathcal{R}(h) = \{ (u,v) \mid v \le -h(u), u\in \mathbb{R}^q \}$, and
assume that $H=\mathcal{R}(h)$. The boundary surface $\partial H $ is
denoted as $ \mathcal{B}(h)=\{(u,v)\mid v=-h(u), u\in \mathbb{R}^q\}$.
For example,
\begin{equation} \label{eq:huexample}
 h(u)=(h_0^2+ u^2/3)^{1/2}
\end{equation}
 with $q=1$, $h_0=0.1$ is shown in Fig~\ref{fig1}.  The region with
$h_0>0$ is related to the confidence limit of the product $\mu_1 \mu_2 $
discussed in \cite{Efron:1985:BCI}, and the region with $h_0\to0$ is
related to the multiple comparisons problem as mentioned later.
Observing $y=(1/\sqrt{2},\sqrt{8/3})=(0.71,1.63)$, say, we would like to
evaluate the chance of $H_0$ being true. We will compute $p$-values by
several methods as shown in Table~\ref{tab1}. Results are also shown for
$y=(3.18,0.20)$.  We occasionally come back to this example throughout
the paper.

\begin{figure}[htb] 
 \includegraphics[width=10cm]{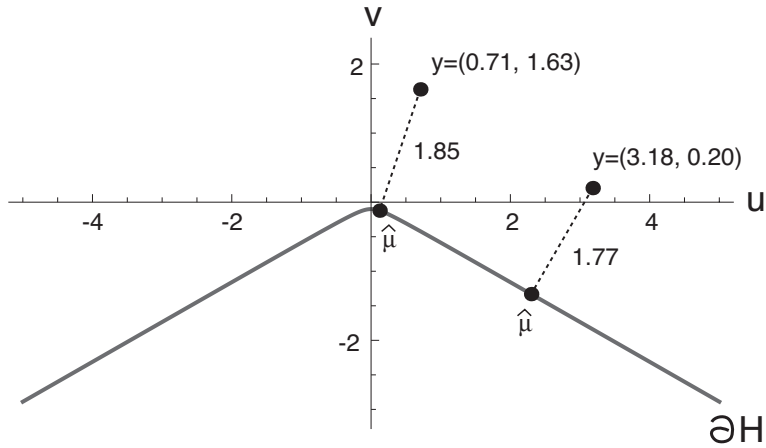} 
\caption{The two cases of observation $y$ and restricted
MLE $\hat\mu(H|y)$.  Signed distances are indicated by dotted lines. The
boundary surface $\partial H$ for (\ref{eq:huexample}) with $h_0=0.1$ is
drawn by solid curve. The null hypothesis is represented as the region
below the curve.} \label{fig1}
\end{figure}

\begin{table}[htb] 
\caption{$p$-values (in percent) computed by several methods.}\label{tab1}
 \begin{tabular}{lccccc}
\hline
observation\quad $y$  & \multicolumn{2}{c}{(0.71, 1.63)} & & \multicolumn{2}{c}{(3.18, 0.20)} \\
\cline{2-3} \cline{5-6}
hypothesis \quad $h_0$ & 0.1 & 0.0 & & 0.1 & 0.0\\
\hline
\multicolumn{6}{c}{conventional testing procedures}\\
LR & 6.4 & 7.5 & & 7.7 & 7.9\\
signed LR & 3.2 & 3.8 & & 3.8 & 3.9\\
$S(y)$ & 18.1 & 20.5 & & 20.8 & 21.0 \\
MCB  & -   & 6.9 & & -   & 6.9\\
\hline
\multicolumn{6}{c}{bootstrap methods}\\
BP  & 1.8 & 2.0 & & 3.8 & 3.8 \\
AU2 & 4.2 & 4.6 & & 3.9 & 3.9 \\
AU3 & 5.5 & 6.2 & & 3.7 & 3.7 \\
DBP & 4.8 & 6.1 & & 3.9 & 4.0\\
DAU & 5.4 & 6.9 & & 3.7 & 3.7\\
\hline
 \end{tabular}
\end{table}

Let us look at likelihood ratio (LR) tests first. We consider null
hypothesis $H_0': \mu\in\partial H$ against alternative $H_1':
\mu\not\in\partial H$.  Since the log-likelihood function is simply
$\ell(\mu; y)=-\tfrac{1}{2}\|y-\mu\|^2$, the maximum likelihood estimate
for $\mu\in\mathbb{R}^{q+1}$ is $y$, and the restricted maximum
likelihood estimate for $\mu\in\partial H$ is given by
\begin{equation} \label{eq:projection}
 \hat \mu(H|y) = \argmin_{\mu\in \partial H} \|y-\mu \|.
\end{equation}
By numerical optimization, we get $\hat\mu(H|y)=(0.12,-0.12)$ for
$y=(0.71,1.63)$, and the LR statistic is then $2\ell(y;
y)-2\ell(\hat\mu(H|y); y)=\| y-\hat\mu(H|y)\|^2 = 1.85^2 = 3.42$.  The
$p$-value is computed as $P(\chi_1^2 \ge 3.42)=0.064$.

However, the following two issues of LR tests are pointed out in
\cite{Efron:1985:BCI} and \cite{Efron:Tibshirani:1998:PR}. (i)~The LR
test ignores the side of $\partial H$ in which $y$ lies.  We can improve
the LR test by replacing the alternative $H_1'$ by $H_1$.
\cite{Mccullagh:1984:LS} introduced the signed LR statistic
$\hat\lambda=\pm \sqrt{2\ell(y; y)-2\ell(\hat\mu(H|y); y)}$ with
positive sign for $y\not \in H$ and negative sign for $y\in H$.
\cite{Efron:1985:BCI} called $\hat\lambda=\pm \| y-\hat\mu(H|y)\|$ as
\emph{signed distance} for the multivariate normal model.  Since
$\hat\lambda \sim N(0,1)$ under $H_0'$ asymptotically, the $p$-value for
testing $H_0'$ against $H_1$ is computed as $1-\Phi(1.85)=0.032$, which
is half of the $p$-value of the LR test.  This one-sided test of
$\hat\lambda$ has twice the power of the (two-sided) LR test.  (ii)~The
LR test and the signed LR test are biased by $\ordi$.  This bias is
corrected by the Bartlett adjustment, which works in a way very similar
to eliminating $\gamma_1$ from (\ref{eq:rejprob-bp}). Our bootstrap
methods will compute $p$-values similar to the bias-corrected singed LR
test.

For testing $H_0$ against $H_1$, we could construct a confidence set of
$\mu$ as
\[
 S(y) = \{\mu \mid \|\mu-y\|^2\le \chi_{2,1-\alpha}^2\},
\]
where $\chi_{2,1-\alpha}^2$ is the upper $\alpha$ point of $\chi_2^2$.
We will reject $H_0$ if the intersection of $S(y)$ and $H$ is empty. The
$p$-value is computed as $P(\chi_2^2 \ge 3.42)=0.181$.  This method
controls the type-I error for any $H$. However, it is very conservative
and $p$-value is unnecessarily large, because $S(y)$ does not take
account of the shape of $H$.

In the case of $h_0=0$, the multiple comparisons with the best (MCB)
procedure of \cite{Hsu:1981:SCI} can be used for testing $H_0$ against
$H_1$.  Observing $\bar x=(\bar x_1,\bar x_2, \bar x_3)$ from $\bar
X\sim N_3(\eta,I_3/n)$ with $\eta=(\eta_1,\eta_2,\eta_3)$, we would like
to know if $\eta_1$ is the largest among the three population means.
MCB assumes the least favorable configuration $\eta_1=\eta_2=\eta_3$ for
computing the null distribution of the test statistic
$t=\sqrt{n}\max(\bar x_2-\bar x_1,\bar x_3-\bar x_1)$.  The null
hypothesis $\eta_1\ge \max(\eta_2,\eta_3)$ is represented as the
cone-shaped region $v\le -|u|/\sqrt{3}$ by transformation
$u=\sqrt{n/2}(\eta_3 - \eta_2)$ and $v=\sqrt{n/6}(\eta_2 + \eta_3 - 2
\eta_1)$. For the two cases of $y$ in Table~\ref{tab1}, the test
statistic is actually the same value $t=2.5$ and $p$-value is $P(T\ge
t)=0.069$. Since MCB is unbiased at $\mu=(0,0)$, i.e., the vertex of the
cone, the $p$-value will be a reasonable value for
$y=(0.71,1.63)$. However, MCB becomes conservative as $\mu$ moves away
from the vertex, and the $p$-value may be unnecessarily large for
$y=(3.18,0.20)$. MCB will be compared with bootstrap methods in the
simulation study of Section~\ref{sec:simulation}.

\section{Bootstrap Methods} \label{sec:methods}

\subsection{Asymptotic theory of surfaces} \label{sec:class-s}

We assume that all the axes in $(u,v)$ coordinates are scaled by
$\sqrt{n}$ asymptotically as $n\to\infty$. This is easily verified for
the spherical example of Section~\ref{sec:intro}.  We only have to
assume that $H$ is represented as $\mathcal{R}(h)$ in a neighborhood of
a point of interest.

We consider the Taylor series of $h(u)$ at $u=0$ as
\begin{equation} \label{eq:hclass-s}
 h(u) \simeq h_0 + h_i u_i + h_{ij} u_i u_j +
h_{ijk} u_i u_j u_k +
h_{ijkl} u_i u_j u_k u_l,
\end{equation}
where $\simeq$ denotes the equality correct up to $\ordiii$ erring
 $\ordiiii$, and the summation convention such as $h_{ij} u_i u_j =
 \sum_{i=1}^q \sum_{j=1}^q h_{ij}u_i u_j$ is used.  Then, the second
 derivative
\[
 h_{ij} = \frac{1}{2} \frac{\partial^2 h(u)}{\partial u_i \partial
 u_j}\Bigr|_0
\]
is $\ordi$, because the numerator is $O(\sqrt{n}\,)$ and the denominator
is $O(n)$.  Similarly, the $k$-th order derivatives are
$O(n^{-(k-1)/2})$, $k\ge2$.  As $n\to\infty$, all these derivatives
approaches zero, and $\partial H$ becomes a flat surface.

We can always assume that $h_0=0$, $h_i=0$ by taking the origin $(0,0)$
at a point on $\partial H$ and the $u_1,\ldots,u_q$ axes in directions
tangent to $\partial H$. These $(u,v)$ coordinates are used in
eq.~(2.10) of \cite{Efron:Tibshirani:1998:PR} for representing $H$.  The
mean curvature of $\partial H$ at $(0,0)$ is defined as
\[
 \gamma_1 = \frac{1}{2} \sum_{i=1}^q
\frac{\partial^2 h(u)}{\partial u_i \partial
 u_i}\Bigr|_0.
\]
The mean curvature of $\partial H$ at $(u,-h(u))$, denoted as
$\gamma_1(h,u)$, is defined similarly by taking the origin there. The
asymptotic expression of $\gamma_1(h,u)$ will be given later in
Section~\ref{sec:geometricquantities}. The mean curvature of the mean
curvature of $\partial H$ at $(0,0)$ is then expressed as
\begin{equation} \label{eq:curvature-of-curvature}
 \beta_3 =  \frac{1}{2} \sum_{i=1}^q
\frac{\partial^2 \gamma_1(h,u)}{\partial u_i \partial
 u_i}\Bigr|_0.
\end{equation}

In the next sections, we will show asymptotic expansions of bootstrap
methods.  It is convenient for the argument there to assume $h_0=\ordz$
and $h_i=\ordii$ by relaxing the assumptions of $h_0=0$ and $h_i=0$.
For $\lambda_0\in \mathbb{R}$, we assume that the observation is
\[
 y=(0,\lambda_0-h_0)
\]
in the $(u,v)$ coordinates. We assume $\lambda_0=\ordz$ for the
local alternatives; in the spherical example, say, $\eta$ approaches
the boundary surface $\|\eta\|=1$ with distance $\ordi$.  Although $u_i$
axes are slightly tilted from the tangent space, the signed distance is
$\hat\lambda = \lambda_0 (1 + O(h_i^2)) \simeq \lambda_0$, meaning that
we can ignore the influence of $h_i$.

We say that a smooth function $h$ belongs to class $\mathcal{S}$ if it
is expressed asymptotically as (\ref{eq:hclass-s}) with coefficients
\begin{equation} \label{eq:classscoef}
h_0=\ordz, h_i=\ordii, h_{ij}=\ordi, h_{ijk}=\ordii,
h_{ijkl}=\ordiii.
\end{equation}
For $h\in \mathcal{S}$, we define the following
quantities representing geometric properties of $\partial H$ at
$(0,-h(0))$,
\begin{equation} \label{eq:gammahij}
\begin{split}
 \gamma_1&=h_{ii}=\ordi,\quad
  \gamma_2=h_{ij}h_{ij}=\ordii,\\
  \gamma_3&=h_{ij} h_{jk}  h_{ki}=\ordiii,\quad
 \gamma_4=h_{iijj}=\ordiii.
\end{split}
\end{equation}
The first three quantities are also written as
$\gamma_1=\Tr(D)$, $\gamma_2=\Tr(D^2)$, $\gamma_3=\Tr(D^3)$ using
$q\times q$ matrix $D$ with elements $(D)_{ij}=h_{ij}$.  
Asymptotic expansions of bootstrap methods will be expressed  up to
$\ordiii$ terms by 
using only
\begin{equation} \label{eq:betabygamma}
\begin{split}
\beta_0 &=\lambda_0=\ordz,\quad
  \beta_1=\gamma_1 - \lambda_0 \gamma_2
 + \tfrac{4}{3} \lambda_0^2  \gamma_3=\ordi, \\
\beta_2&=3\gamma_4 -
 \gamma_1 \gamma_2 -\tfrac{4}{3} \gamma_3=\ordiii,\quad
 \beta_3 = 6\gamma_4 - 2\gamma_1\gamma_2 - 4\gamma_3 = \ordiii.
\end{split}
\end{equation}
We will verify in Section~\ref{sec:geometricquantities} that the above
definition of $\beta_3$ in (\ref{eq:betabygamma}) is consistent with
(\ref{eq:curvature-of-curvature}).

\subsection{Asymptotic expansion of the bootstrap probability} 
\label{sec:asymptotic-bp}

\cite{Efron:Tibshirani:1998:PR} showed the asymptotic expansion of
$\mathrm{BP}(H|y)$ up to $\ordii$ terms.  We generalize their
eq.~(2.19) to include $\ordiii$ terms. For convenience, we use
\[
 \bar\Phi(x) = 1-\Phi(x) = \Phi(-x).
\]
All the proofs of theorems are found in Appendix.
\begin{theo}[Bootstrap probability] \label{thm:bp}
Consider $y=(0,\lambda_0-h_0)$ and the region $H=\mathcal{R}(h)$ for
$h\in \mathcal{S}$.  The bootstrap probability for $\sigma^2=1$ is
then expressed asymptotically as
\begin{equation} \label{eq:bpexpgamma}
\mathrm{BP}(H|y) \simeq \bar\Phi\Bigl[
\lambda_0 + \gamma_1  - \lambda_0 \gamma_2 +
3 \gamma_4 - \gamma_1 \gamma_2 -\tfrac{4}{3}(1-\lambda_0^2) \gamma_3
\Bigr].
\end{equation}
Using the coefficients defined in (\ref{eq:betabygamma}), it becomes
\begin{equation}\label{eq:bpexpbeta}
\mathrm{BP}(H|y) 
\simeq
\bar\Phi(\beta_0 + \beta_1 + \beta_2).
\end{equation}
\end{theo}

\cite{Efron:Tibshirani:1998:PR} also showed a third-order accurate
$p$-value. We generalize their eq.~(2.17) to include $\ordiii$ terms.
We will show later in Section~\ref{sec:asymptotic-pv} that the $p$-value
defined below is fourth-order accurate.
\begin{equation}\label{eq:pvalexpbeta}
 \mathrm{PV}(H|y) \simeq
\bar\Phi(
\beta_0 - \beta_1 - \beta_2 +  \beta_3 
).
\end{equation}
Comparing (\ref{eq:bpexpbeta}) with (\ref{eq:pvalexpbeta}), we find that
BP differs from PV by $\ordi$ and so BP is only first-order accurate in
general.

For simplifying geometric argument, here we assume $h_0=h_i=0$ and
$y=(0,\lambda_0)$ by taking the origin of the coordinates at
$\hat\mu(H|y)$. Then the signed distance is $\hat\lambda=\lambda_0$, and
the geometric quantities, such as the mean curvature $\gamma_1$, are now
defined at $\hat\mu(H|y)=(0,0)$. Then the two geometric quantities,
$\lambda_0$ and $\gamma_1$, determine the $p$-value of singed
$\mathrm{LR} = \bar\Phi(\lambda_0)$, $\mathrm{BP}=\bar\Phi(\lambda_0
+\gamma_1) + \ordii$, and $\mathrm{PV}=\bar\Phi(\lambda_0 -\gamma_1) +
\ordii$ up to $\ordi$ terms.  For $\gamma_1>0$, they are ordered as BP
$<$ signed LR $<$ PV, and so $P(\mathrm{BP} < \alpha)$ will be larger
than $P(\mathrm{PV} < \alpha)\simeq \alpha$.  This confirms
(\ref{eq:rejprob-bp}), where $\gamma_1$ is defined at $\mu$ instead of
$\hat\mu$ though.

Let us look at the numerical example of $y=(0.71,1.63)$ with $h_0=0.1$
in Table~\ref{tab1}. We know $\gamma_1$ is positive by looking at the
convex shape of $H$, and $\mathrm{BP}=0.018$ is, in fact, smaller than
signed LR$=0.032$.  From these two values, the mean curvature can be
estimated by
\[
 \gamma_1 = \bar\Phi^{-1}(\mathrm{BP}) -
 \bar\Phi^{-1}(\mathrm{signed\,LR}) + \ordii,
\]
which gives $\gamma_1 \approx \bar\Phi^{-1}(0.018) -
\bar\Phi^{-1}(0.032) = 2.08-1.85=0.23$ at $\hat \mu=(0.12,-0.12)$.  We
can then compute PV up to $\ordi$ terms as $\mathrm{PV}\approx
\bar\Phi(1.85 - 0.23) = 0.053$, which is close to AU3, DBP, and DAU
explained in the next sections.  On the other hand, the mean curvature
$\gamma_1 \approx 0.002$ is much smaller at $\hat\mu=(2.30,-1.33)$ for
$y=(3.18,0.20)$, and $\mathrm{PV}\approx 0.038$ is not different from
$\mathrm{BP} = 0.038$; BP does not need bias correction and all the
bootstrap methods are very close to the signed LR in Table~\ref{tab1}.

\cite{Efron:1985:BCI} and \cite{Efron:Tibshirani:1998:PR} computed PV up
to $\ordi$ terms in the same way as above but using only bootstrap
probabilities. Their bias-corrected (BC) bootstrap method estimates the
mean curvature by
\[
 \gamma_1 = \bar\Phi^{-1}\Bigl(\mathrm{BP}(H|\hat\mu(H|y))\Bigr)
+ \ordiii,
\]
which is verified by letting $\lambda_0=0$ in (\ref{eq:betabygamma}) and
(\ref{eq:bpexpbeta}).  In the next sections, we attempt computing PV up
to higher-order terms using only bootstrap probabilities.

\subsection{Multiscale bootstrap}

For adjusting the bias of BP, we would like to express
$\mathrm{BP}_{\sigma^2}$ as a function of $\sigma^2$.
\cite{Shimodaira:2002:AUT,Shimodaira:2004:AUT} showed the asymptotic
expansion of $\mathrm{BP}_{\sigma^2}(H|y)$ up to $\ordii$ terms.  Here
we include $\ordiii$ terms to it. This is an immediate consequence of
Theorem~\ref{thm:bp} via a rescaling argument.
\begin{theo}[Scaling-law of the bootstrap probability] \label{thm:bps}
 For the $H$ and $y$ given in Theorem~\ref{thm:bp}, the bootstrap
 probability for $\sigma^2>0$ is expressed as
\begin{equation} \label{eq:rescaleregion}
 \mathrm{BP}_{\sigma^2}(H|y) =  \mathrm{BP}(\sigma^{-1}H|\sigma^{-1}y),
\end{equation}
where $ \sigma^{-1} H=\{\sigma^{-1} y : y \in H \}$.
By replacing 
\begin{equation} \label{eq:rescalebeta}
\beta_0 \to \sigma^{-1}\beta_0,\quad
\beta_1 \to \sigma \beta_1,\quad
\beta_2 \to \sigma^3 \beta_2
\end{equation}
in (\ref{eq:bpexpbeta}), the right hand side of (\ref{eq:rescaleregion})
 is expressed asymptotically as
\begin{equation} \label{eq:bpsexp}
 \mathrm{BP}_{\sigma^2}(H|y) \simeq
\bar\Phi\Bigl[
 \beta_0 \sigma^{-1} +\beta_1  \sigma + \beta_2 \sigma^3 
\Bigr].
\end{equation}
\end{theo}

\cite{Shimodaira:2008:TRN} introduced the \emph{normalized bootstrap
probability} defined by
\[
 \mathrm{NBP}_{\sigma^2}(H|y) = \Phi\bigl[ \sigma \Phi^{-1}(
 \mathrm{BP}_{\sigma^2}(H|y) ) \bigr]
\]
for $\sigma^2>0$, and considered an ``approximately unbiased'' $p$-value
defined formally by
\[
 \mathrm{AU}(H|y) = \mathrm{NBP}_{-1}(H|y).
\]
For extrapolating $\mathrm{NBP}_{\sigma^2}$ to $\sigma^2\le0$, we use
 the scaling-law of BP. It follows from Theorem~\ref{thm:bps} that the
 normalized bootstrap probability is expressed asymptotically as
\begin{equation} \label{eq:nbps}
 \mathrm{NBP}_{\sigma^2}(H|y) \simeq
\bar\Phi\Bigl[
 \beta_0 +  \beta_1 \sigma^2 + \beta_2 \sigma^4 
\Bigr]
\end{equation}
for $\sigma^2>0$, and it is extrapolated to $\sigma^2\le 0$ by the
right-hand side of (\ref{eq:nbps}). In particular for $\sigma^2=-1$, we
obtain the asymptotic expansion of AU as
\begin{equation} \label{eq:auexp}
\mathrm{AU}(H|y)  \simeq
\bar\Phi( \beta_0 -  \beta_1 + \beta_2 ).
\end{equation}
Comparing (\ref{eq:auexp}) with (\ref{eq:pvalexpbeta}), we find that $
\mathrm{AU}(H|y) = \mathrm{PV}(H|y) + \ordiii$, indicating AU is
third-order accurate in general. The remaining bias of order $\ordiii$
comes from the difference $\bar\Phi^{-1}(\mathrm{AU})-
\bar\Phi^{-1}(\mathrm{PV}) \simeq \tfrac{4}{3}\gamma_3$.

In complicated applications, we do not know the values of the
coefficients $\beta_0$, $\beta_1$, $\beta_2$, or they are just hardly
obtained through mathematical analysis. In the multiscale bootstrap of
\cite{Shimodaira:2008:TRN}, we estimate $\beta_0$, $\beta_1$, $\beta_2$
by fitting the right-hand side of (\ref{eq:bpsexp}) to observed values
of $\mathrm{BP}_{\sigma^2}(H|y)$ computed for several $\sigma^2>0$ values,
say, $\sigma^2_1,\ldots,\sigma^2_S$. This is equivalent to fitting
quadratic model $\beta_0+\beta_1 \sigma^2 + \beta_2 (\sigma^2)^2$ in
terms of $\sigma^2$ to observed values of $\sigma \bar\Phi^{-1}(
\mathrm{BP}_{\sigma^2}(H|y) )$.  Using the estimated values of the
coefficients, we can compute (\ref{eq:nbps}) for $\sigma^2 \le 0$.  In
the original form of multiscale bootstrap of \cite{Shimodaira:2002:AUT},
only two coefficients $\beta_0, \beta_1$ are estimated by linear model
$\beta_0+\beta_1 \sigma^2$, and $p$-value is computed as
$\mathrm{AU}=\bar\Phi( \beta_0 - \beta_1)$. The difference of the two AU
values is only $\ordiii$ and both the AU values are third-order
accurate.

The procedure is illustrated in Fig~\ref{fig2} for the numerical example of
$y=(0.71,1.63)$ with $h_0=0$, where the geometric quantities are
actually not defined at the vertex $\mu=(0,0)$.  We plotted $\sigma
\bar\Phi^{-1}( \mathrm{BP}_{\sigma^2}(H|y) )$ in a solid curve for $0.1
< \sigma^2 < 1.9$, instead of plotting the values for
$\sigma^2_1,\ldots,\sigma^2_S$.  We denote $\mathrm{AU}k$ when
extrapolation to $\sigma^2\le0$ is made by Taylor expansion with $k$
terms at $\sigma^2=1$.  This computes
$\mathrm{AU}2=\bar\Phi(1.69)=0.046$ by the linear model, and
$\mathrm{AU}3=\bar\Phi(1.53)=0.062$ by the quadratic model.
Interestingly, the procedure behaves similarly to the case of $h_0=0.1$,
and it seems working fine even when $h_0=0$ as will be seen also in
the simulation study of Section~\ref{sec:simulation}.

\begin{figure}[htb] 
\includegraphics[width=10cm]{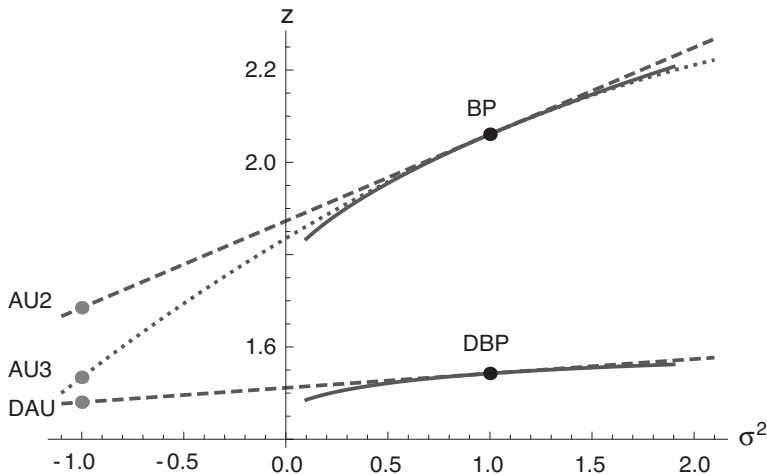}
\caption{Illustration of multiscale bootstrap and
  multiscale-double bootstrap for $y=(0.71,1.63)$. The boundary surface
  $\partial H$ is defined by (\ref{eq:huexample}) with $h_0=0$. Vertical
  axis indicates $z=\bar\Phi^{-1}(p)$ for several $p$-values. In
  multiscale bootstrap, $z=\sigma \bar\Phi^{-1}(
  \mathrm{BP}_{\sigma^2}(H|y) )$ is extrapolated to $\sigma^{2}=-1$ by
  linear model (dashed line) or quadratic model (dotted curve). In
  multiscale-double bootstrap, $z=\bar\Phi^{-1}(
  \mathrm{DBP}_{1,\sigma^2}(H|y) )$ is extrapolated to $\sigma^2=-1$ by
  linear model (dashed line).} \label{fig2}
\end{figure}

\subsection{Multiscale-double bootstrap}

The bias of BP can also be adjusted by the iterated bootstrap.  Instead
of (\ref{eq:yboot}), we generate many bootstrap replicates around $\hat
\mu(H|y)$ by
\[
 Y^{+} \sim N_{q+1}(\hat \mu(H|y), \tau^2 I_{q+1})
\]
for some $\tau^2>0$. The notation $Y^+$ is used to make the distinction
clear. For each generated value of $y^+$, we compute
$\mathrm{BP}_{\sigma^2}(H|y^+)$. This involves second-level bootstrap
and huge computation.  We calibrate $ \mathrm{BP}_{\sigma^2}(H|y)$ by
the distribution of $\mathrm{BP}_{\sigma^2}(H|Y^+)$. The double
bootstrap probability of $H$ for a given $y$ is defined as
\begin{equation} \label{eq:dbptausigma}
 \mathrm{DBP}_{\tau^2,\sigma^2}(H|y)
=P_{\tau^2} \Bigl[ \mathrm{BP}_{\sigma^2}(H|Y^{+})
 \le \mathrm{BP}_{\sigma^2}(H|y) \mid \hat\mu(H|y)   \Bigr].
\end{equation}
The variances are usually $\sigma^2=\tau^2=1$ and we simply denote DBP
or $\mathrm{DBP}(H|y)$ for $\mathrm{DBP}_{1,1}(H|y)$.
\cite{Efron:Tibshirani:1998:PR} called DBP as a calibrated confidence
level and mentioned that DBP is third-order accurate.

We will show later in Section~\ref{sec:asymptotic-dbp} that the double
 bootstrap probability for $\sigma^2>0$, $\tau^2=1$ is expressed
 asymptotically as
\begin{equation} \label{eq:dbpexps}
 \mathrm{DBP}_{1,\sigma^2}(H|y) \simeq
\bar\Phi\Bigl[ \beta_0  - \beta_1 - \beta_2 - \beta_3  \sigma^2 \Bigr],
\end{equation}
and it is extrapolated to $\sigma^2\le 0$ by the right-hand side.
Comparing (\ref{eq:dbpexps}) with (\ref{eq:pvalexpbeta}), we find that $
\mathrm{DBP}_{1,\sigma^2}(H|y) = \mathrm{PV}(H|y) + \ordiii$.  
In particular for $\sigma^2=1$, we confirm
that DBP is third-order accurate.

The remaining bias of order $\ordiii$ in DBP comes from the difference
\[
 \bar\Phi^{-1}(\mathrm{DBP}_{1,\sigma^2})- \bar\Phi^{-1}(\mathrm{PV})
\simeq -(1+\sigma^2)\beta_3,
\]
which vanishes when $\sigma^2=-1$. The bias-corrected DBP is 
defined formally by
\[
 \mathrm{DAU}(H|y) = \mathrm{DBP}_{1,-1}(H|y)
\]
so that DAU is forth order accurate.  Another advantage of DAU over DBP
is robustness against computational error of $\hat\mu(H|y)$ as mentioned
in Section~\ref{sec:asymptotic-dbp}.  The name of DAU may be understood
in the interpretation
\[
  \mathrm{DAU}(H|y)
\simeq P \Bigl[ \mathrm{AU}(H|Y^{+})
 \le \mathrm{AU}(H|y) \mid \hat\mu(H|y)   \Bigr],
\]
which immediately follows from (\ref{eq:dbptausigma}) by considering the
equivalence of contour surfaces of $\mathrm{BP}_{\sigma^2}(H|y)$ and
$\mathrm{NBP}_{\sigma^2}(H|y)$ as mentioned just before
Lemma~\ref{lem:additivity} in Section~\ref{sec:contourbp}.

Similarly to the computation of AU, we estimate the coefficients
$\beta_0-\beta_1-\beta_2$ and $\beta_3$ by fitting a linear model to
observed values of $ \bar\Phi^{-1}(\mathrm{DBP}_{1,\sigma^2})$. The
procedure is illustrated in Fig~\ref{fig2}.  We plotted
$\bar\Phi^{-1}(\mathrm{DBP}_{1,\sigma^2})$ in a solid curve for
$0.1<\sigma^2<1.9$ and extrapolation to $\sigma^2=-1$ is
made by Taylor expansion at $\sigma^2=1$.  $\mathrm{DAU} = \bar
\Phi(1.48)=0.069$ is slightly larger than $\mathrm{DBP} = \bar
\Phi(1.54)=0.061$ in this example.

\subsection{Simulation study} \label{sec:simulation}

Rejection probabilities (\ref{eq:rejprob-bp}), (\ref{eq:rejprob-dbp}),
and those for other approximate $p$-values are shown in
Table~\ref{tab2}. The region $H$ is the cone-shaped region mentioned in
Section~\ref{sec:lrtest}, where $h$ is specified by (\ref{eq:huexample})
with $h_0=0$.  Rejection probabilities are computed for several
$\mu=(u,-h(u))$ on $\partial H$.  These values are computed accurately
by numerical integration instead of Monte-Carlo simulation for avoiding
sampling error.  Looking at the table, we verify that MCB is unbiased at
$u=0$. However, the rejection probability of MCB is much smaller than
$\alpha$ for larger $u$.

All the bootstrap methods behave similarly in the sense that the bias is
large at $u=0$ and the bias decreases as $u$ becomes larger. BP has the
largest bias, and all the bias-corrected bootstrap probabilities have
smaller bias. In particular, AU3, DBP, and DAU have very small bias.
The difference between DBP and DAU is small, but DAU performs better
than DBP at all $u$ values.  Interestingly, the bias correction methods
work fine, even though $h(u)$ is not smooth at $u=0$.  Looking at
Table~\ref{tab1} again, we confirm that AU3, DBP, DAU values are close to
MCB for $y=(0.71, 1.63)$, agreeing with the simulation at $u=0$.

\begin{table}[htb] 
\caption{Rejection probabilities (in percent) at significance level
$\alpha=5\%$. }\label{tab2}
\begin{tabular}{lccccccc}
\hline
$u$ & 0.0 & 0.5 & 1.0 & 1.5 & 2.0 & 2.5 & 3.0 \\
\hline
BP  & 13.39 & 8.894 & 6.678 & 5.676 & 5.253 & 5.086 & 5.027 \\
AU2 & 7.655 & 5.171 & 4.459 & 4.447 & 4.628 & 4.801 & 4.912 \\
AU3 & 6.609 & 4.718 & 4.493 & 4.746 & 4.982 & 5.080 & 5.081 \\
DBP & 6.619 & 4.590 & 4.202 & 4.364 & 4.610 & 4.795 & 4.905 \\
DAU & 6.476 & 4.660 & 4.481 & 4.746 & 4.981 & 5.084 & 5.092 \\
MCB & 5.000 & 3.340 & 2.880 & 2.783 & 2.768 & 2.766 & 2.766 \\
\hline
\end{tabular}
\end{table}

\section{Geometry of smooth surfaces} \label{sec:geometry}

In this section, we discuss only geometry of smooth surfaces via simple
but tedious calculation without any probability argument. The results
will be used in Section~\ref{sec:asymptotic-analysis} for deriving
asymptotic accuracy of the bootstrap methods.  We work on the region
$H=\mathcal{R}(h)$ and boundary surface $\partial H=\mathcal{B}(h)$ for
$h\in \mathcal{S}$ expressed in the $(u,v)$ coordinates.

\subsection{Representing surfaces in local coordinates} 
\label{sec:localcoordinates}

We consider local coordinates $(\Delta u, \Delta v)$ with $\Delta
u=(\Delta u_1,\ldots, \Delta u_q)\in \mathbb{R}^q$ and $\Delta v\in
\mathbb{R}$ by taking the origin at $(u,-h(u))$. A point $(\Delta u,
\Delta v)$ is expressed in the $(u,v)$ coordinates as
\begin{equation} \label{eq:uvlocaluv}
 (u,-h(u)) + \Delta u_i\, b_i + \Delta v\, \|f\|^{-1} f
\end{equation}
using basis $\{b_1,\ldots, b_q, f\}$ in $\mathbb{R}^{q+1}$ defined as
follows. 

Here $\|f\|=\sqrt{f\cdot f}$ is the norm of $f\in \mathbb{R}^{q+1}$ with
the inner product $a\cdot b=\sum_{i=1}^{q+1}a_i b_i$ for two vectors
$a,b\in\mathbb{R}^{q+1}$.  We denote $\delta_i =
(\delta_{i1},\ldots,\delta_{iq}) \in \mathbb{R}^{q}$ with the
Kronecker delta $\delta_{ij}$, and $\nabla=(\partial/\partial
u_1,\ldots,\partial/\partial u_q)$.  Then
\[
b_i = \Bigl(\delta_i, -\frac{\partial h}{\partial u_i}\Bigr),
i=1,\ldots,q,
\]
are tangent to $\partial H$ at $(u,-h(u))$, and
the normal vector 
\[
 f=(\nabla h,1)
\]
satisfies $f \cdot b_i=0$, meaning that $f$ is orthogonal to $\partial
H$ at $(u,-h(u))$.  The vectors $b_i$ and $f$ should be denoted as
$b_i(u)$ and $f(u)$, but the dependence on $u$ is suppressed in the
notation.

\begin{lemm} \label{lem:hlocal}
For $h\in \mathcal{S}$,  the region $H=\mathcal{R}(h)$ is expressed in
the $(\Delta u,\Delta v)$ coordinates at $(u,-h(u))$ as
\[
 H = \{(\Delta u,
\Delta v) \mid \Delta v \le -\tilde
h(\Delta u), \Delta u \in \mathbb{R}^q  \}
\]
with $\tilde h\in\mathcal{S}$.  The coefficients are $\tilde h_0=\tilde
h_i=0$, $\tilde h_{ij} = h_{ij} + 3 h_{ijk} u_k + (6 h_{ijkl} -2 h_{ij}
h_{mk} h_{ml}) u_k u_l$, $ \tilde h_{ijk} = h_{ijk} + 4 h_{ijkl} u_l -
\tfrac{4}{3} (h_{ij} h_{km} h_{ml} + h_{ik} h_{jm} h_{ml} + h_{jk}
h_{im} h_{ml} ) u_l $, $\tilde h_{ijkl} = h_{ijkl}$.
 
\end{lemm} 

\subsection{Expressions of the four geometric quantities} 
\label{sec:geometricquantities}

We consider an orthonormal basis $\{c_1,\ldots, c_q, \|f\|^{-1} f\}$ for
the local coordinates at $(u,-h(u))$, where $\{c_1,\ldots,c_q\}$ is an
arbitrary orthonormal basis of the tangent space; $c_i \cdot c_j =
\delta_{ij}$ and $c_i \cdot f = 0$.  The dependence of these vectors on
$u$ is suppressed in the notation again.  A point $(x,\Delta v)$ with
$x=(x_1,\ldots,x_q)\in\mathbb{R}^q$ and $\Delta v\in\mathbb{R}$
corresponds to
\[
 (u,-h(u)) + \Delta u_i\, c_i + \Delta v\, \|f\|^{-1} f
\]
in the $(u,v)$ coordinates.

In the $(x,\Delta v)$ coordinates, $\partial H$ is expressed as $\Delta
v = -d(x)$ with
\[
 d(x) \simeq d_{ij} x_i x_j  +  d_{ijk} x_i x_j x_k + d_{ijkl} x_i x_j
x_k x_l.
\]
Then we apply the definitions of $\gamma_i$ in (\ref{eq:gammahij}) to
$d(x)$ as follows.
\[
\begin{split}
 \gamma_1(h,u) &= d_{ii}=\Tr(D),\quad
 \gamma_2(h,u) = d_{ij} d_{ij}=\Tr(D^2),\\
 \gamma_3(h,u) &= d_{ij} d_{jk} d_{ki}=\Tr(D^3),\quad
 \gamma_4(h,u) = d_{iijj},
\end{split} 
\]
where $D$ is $q \times q $ matrix with elements $(D)_{ij}=d_{ij}$.  The
four geometric quantities are invariant to the choice of orthonormal
basis as will be seen in (\ref{eq:uvgamma}) below. 

\begin{lemm} \label{lem:glocal}
For $h\in \mathcal{S}$, we consider the local coordinates $(\Delta u,
\Delta v)$ at $(u,-h(u))$ using the basis $\{b_1,\ldots,b_q, 
f\}$.  Let $G$ be $q \times q$ matrix with elements $(G)_{ij} = g_{ij} =
b_i \cdot b_j$ for $i,j=1,\ldots,q$, and $g^{ij} = (G^{-1})_{ij}$ be the
elements of the inverse matrix of $G$.  Then the four geometric
quantities are expressed as
\begin{equation} \label{eq:uvgamma}
 \begin{split}
  \gamma_1(h,u) &= \tilde h_{ij} g^{ij}=\Tr(\tilde D G^{-1}),\quad
  \gamma_2(h,u) = \tilde h_{ij} g^{jk} \tilde h_{kl} g^{li}
=\Tr((\tilde D G^{-1})^2),\\
  \gamma_3(h,u) &= \tilde h_{ij} g^{jk} \tilde h_{kl} g^{lm} \tilde
  h_{mn} g^{ni}=\Tr((\tilde D G^{-1})^3),\quad
  \gamma_4(h,u) = \tilde h_{ijkl} g^{ij} g^{kl}
 \end{split}
\end{equation}
using the coefficients $\tilde h_{ij}$ and $\tilde h_{ijkl}$ defined in
Lemma~\ref{lem:hlocal} and $q \times q$ matrix $\tilde D$ with elements
$(\tilde D)_{ij}=\tilde h_{ij}$.  They are expressed asymptotically as
\begin{equation} \label{eq:hijgamma}
 \begin{split}
 \gamma_1(h,u) & \simeq h_{ii} + 3 h_{iik} u_k + (6 h_{iikl} -2 h_{ii}
 h_{mk} h_{ml} -4 h_{ij} h_{ik} h_{jl}) u_k u_l,\\
 \gamma_2(h,u) &\simeq h_{ij} h_{ij} + 6 h_{ij} h_{ijk} u_k,\quad
 \gamma_3(h,u)\simeq h_{ij} h_{jk} h_{ki},\quad
\gamma_4(u,h) \simeq h_{iijj}
 \end{split}
\end{equation}
using the coefficients of $h(u)$.  In particular,
$\gamma_i=\gamma_i(h,0)$, $i=1,\ldots,4$, are consistent with their
 definitions in (\ref{eq:gammahij}). Also,
\[
\frac{1}{2} \frac{\partial^2 \gamma_1(h,u)}{\partial u_i \partial
 u_j}\Bigr|_0 \simeq
6 h_{mmij} -2 h_{mm} h_{li} h_{lj} -4 h_{ml} h_{mi} h_{lj}
\]
confirms that the definition of $\beta_3$ in (\ref{eq:betabygamma}) is
consistent with (\ref{eq:curvature-of-curvature}).
\end{lemm}

\subsection{Shifting surfaces} \label{sec:twosurfaces}

We consider shifting $\mathcal{B}(h)$ toward the normal direction.  Let
$f(u)$ be the normal vector at $(u,-h(u))\in \mathcal{B}(h)$. For a
specified $\lambda\in\mathcal{S}$, we move the point $(u,-h(u))$ by
$\lambda(u)$ toward the normal direction.  This is expressed as
\begin{equation} \label{eq:sdefinition}
(\theta,-s(\theta)) = (u,-h(u)) + \lambda(u) \|f(u)\|^{-1} f(u),
\end{equation}
where $s(u)$ is some function of $u\in\mathbb{R}^q$, and
$\theta\in\mathbb{R}^q$ is used when distinction is needed.  We can
interpret (\ref{eq:sdefinition}) as
\[
 \hat\mu(H |(\theta,-s(\theta)) ) = (u,-h(u))
\]
with signed distance $\lambda(u)$.  For sufficiently large $n$, such
$s(\theta)$ is uniquely defined for each $\theta$, because all the
surfaces approach flat as $n\to\infty$.  We denote
(\ref{eq:sdefinition}) as
\[
 s=\mathcal{M}(h,\lambda).
\]

\begin{lemm}\label{lem:twosurfaces}
Let $s=\mathcal{M}(h,\lambda)$ for $h\in \mathcal{S}$, $\lambda \in \mathcal{S}$.  If $\lambda(u)$
is expressed as
\[
 \lambda(u) \simeq \lambda_0 + \lambda_i u_i + \lambda_{ij} u_i u_j
\]
with $\lambda_0=\ordz$, $\lambda_i=\ordii$, $\lambda_{ij}=\ordiii$, then
we have $s\in \mathcal{S}$ with coefficients $ s_0=h_0-\lambda_0 =
\ordz$, $ s_i = h_i-\lambda_i - 2\lambda_0 h_{mi} (h_m - \lambda_m) =
\ordii$, $ s_{ij} = h_{ij}-\lambda_{ij} - 2\lambda_0 h_{mi} h_{mj} +
4\lambda_0^2 h_{ml} h_{mi} h_{lj} = \ordi$, $s_{ijk} = h_{ijk} -
2\lambda_0 (h_{mi} h_{mjk}+h_{mj} h_{mik}+h_{mk} h_{mij} ) = \ordii$, $
s_{ijkl} = h_{ijkl} =\ordiii$.  The four geometric quantities at
$(0,-s(0))$ are $ \gamma_1(s,0) = s_{ii} \simeq \gamma_1 -
\lambda_{ii} -2\lambda_0 \gamma_2 + 4\lambda_0^2 \gamma_3$,
$\gamma_2(s,0) = s_{ij} s_{ij} \simeq \gamma_2 - 4\lambda_0
\gamma_3$, $\gamma_3(s,0) = s_{ij} s_{jk} s_{ki} \simeq \gamma_3$,
$\gamma_4(s,0) = s_{iijj} \simeq \gamma_4$, where
$\gamma_i=\gamma_i(h,0)$, $i=1,\ldots,4$.
\end{lemm}

\section{Asymptotic analysis of bootstrap methods}
\label{sec:asymptotic-analysis}

We are going to show the asymptotic expansions of PV and DBP, and then
prove the asymptotic accuracy of the bootstrap methods. The argument is
based on the geometric tools developed in Section~\ref{sec:geometry} as
well as another tool to be developed below.

\subsection{Contour surfaces of bootstrap probability}
\label{sec:contourbp}

We consider a surface on which the bootstrap probability remains
constant.  For $H=\mathcal{R}(h)$ with $h\in\mathcal{S}$, we consider a
function $s(u)$ of $u\in \mathbb{R}^q$ satisfying
\[
 \mathrm{BP}_{\sigma^2}(H | (u,-s(u)) ) = 
1-\alpha,\quad u\in\mathbb{R}^q,
\]
meaning $ \mathrm{BP}_{\sigma^2}(H | y ) =1-\alpha$ is constant for any
$y\in\mathcal{B}(s)$.  Then, $\mathcal{B}(s)$, as well as $s$ itself,
will be called as the \emph{contour surface of the bootstrap
probability} of $H$ with variance $\sigma^2>0$ at level $1-\alpha$.  In
particular, we choose $\alpha$ so that $(0,\lambda_0-h_0)\in
\mathcal{B}(s)$ for a specified $\lambda_0\in \mathbb{R}$.  We denote
this contour surface as
\[
 s = \mathcal{L}_{\sigma^2}(h,\lambda_0).
\]

\begin{lemm} \label{lem:contourbp}
Let $ s = \mathcal{L}_{\sigma^2}(h,\lambda_0)$ for $h\in\mathcal{S}$,
 $\lambda_0\in \mathbb{R}$, and $\sigma^2>0$. Then, 
$s$ is expressed as
 $s=\mathcal{M}(h,\lambda) $ by specifying $\lambda(u) \simeq \lambda_0
 + \lambda_i u_i + \lambda_{ij} u_i u_j$ with $\lambda_0=\ordz$,
\begin{equation}
\lambda_i =  \sigma^2 ( -3 h_{mmi} + 6\lambda_0 h_{ml}
  h_{mli}),\quad  \lambda_{ij}=
\sigma^2 (-6
 h_{mmij} +2 h_{mm} h_{li} h_{lj} + 4 h_{ml} h_{mi} h_{lj}). 
\label{eq:lambdaucontour}
\end{equation}
We have $s\in\mathcal{S}$
with coefficients
\begin{equation}\label{eq:contourcoef}
\begin{split}
 s_0 &= h_0 - \lambda_0,\quad
 s_i = h_i - 2\lambda_0  h_m h_{mi}+
 \sigma^2( 3h_{mmi}-6 \lambda_0 h_{ml} h_{mli}   - 6\lambda_0
 h_{mi}h_{mll} ),\\
 s_{ij} &= h_{ij} - 2\lambda_0 h_{mi} h_{mj} + 4\lambda_0^2 
h_{ml}h_{mi} h_{lj} 
+\sigma^2(6h_{ijmm} -2 h_{mm}h_{li}h_{lj} -4 h_{ml}h_{mi} h_{lj}),\\
 s_{ijk} &= h_{ijk} - 2\lambda_0 (h_{mi} h_{mjk}+h_{mj} h_{mik}+h_{mk}
 h_{mij} ),\quad
s_{ijkl} = h_{ijkl}.
\end{split} 
\end{equation}
The four geometric quantities of $s$ at $(0,-s(0))$ are
\begin{equation}\label{eq:contourgamma}
 \begin{split}
 \gamma_1(s,0) &\simeq\gamma_1 - 2\lambda_0 \gamma_2 + 4\lambda_0^2 \gamma_3
+ \sigma^2(6 \gamma_4 - 2\gamma_1 \gamma_2 - 4\gamma_3),\\
\gamma_2(s,0) &\simeq \gamma_2 - 4\lambda_0 \gamma_3,\quad
\gamma_3(s,0)\simeq \gamma_3,\quad
\gamma_4(s,0)\simeq \gamma_4,
 \end{split}
\end{equation}
where $\gamma_i=\gamma_i(h,0)$, $i=1,\ldots,4$.
\end{lemm}

We denote the $\lambda(u)$ of (\ref{eq:lambdaucontour}) as
$\lambda_{\sigma^2}(u) = \lambda_0 - \sigma^2 \kappa(u)$
with
\begin{align}
 \kappa(u) &= 
\gamma_1(h,u)-\gamma_1(h,0) - \lambda_0 (\gamma_2(h,u)-\gamma_2(h,0) )
\notag\\
&\simeq ( 3 h_{mmi} - 6\lambda_0 h_{ml} h_{mli}) u_i + (6
 h_{mmij} - 2 h_{mm} h_{li} h_{lj} - 4 h_{ml} h_{mi} h_{lj}) 
u_i u_j.\label{eq:kappa}
\end{align}
This also relates to (\ref{eq:curvature-of-curvature}) as
$(1/2)\partial^2 \kappa(u)/\partial u_i \partial u_i|_0 = \beta_3$ or
$(1/2)\partial^2 \lambda_{\sigma^2}(u)/\partial u_i \partial u_i|_0 = -
\sigma^2\beta_3$.  The contour surface of $\mathrm{BP}_{\sigma^2}(H|y)$
for $\sigma^2>0$ is expressed asymptotically as
\[
 \mathcal{L}_{\sigma^2}(h,\lambda_0) = \mathcal{M}(h,\lambda_{\sigma^2}),
\]
and it is extrapolated formally to $\sigma^2\le 0$ by the right-hand
side. It becomes the surface with constant signed distance
$\lambda(u)=\lambda_0$ when $\sigma^2=0$. For $\sigma^2\in\mathbb{R}$,
the deviation $\lambda_{\sigma^2}(u)-\lambda_0 = -\sigma^2 \kappa(u)$ is
proportional to $\sigma^2$.  Therefore, the formal definition of
$\mathcal{L}_{\sigma^2}(h,\lambda_0)$ for $\sigma^2<0$ makes sense, at
least, in terms of computation, although $\mathrm{BP}_{\sigma^2}(H|y)$
is not defined. In fact, $\mathcal{L}_{\sigma^2}(h,\lambda_0)$ is
interpreted as the contour surface of $\mathrm{NBP}_{\sigma^2}(H|y)$ for
$\sigma^2\in\mathbb{R}$, because we will get the same expression of
$\lambda_{\sigma^2}(u)$ for $\mathrm{NBP}_{\sigma^2}(H|y)=1-\alpha'$ by
substituting $\sigma z_\alpha = z_{\alpha'}$ in the proof of
Lemma~\ref{lem:contourbp}.

\begin{lemm} \label{lem:additivity}
Two functions $h,s\in \mathcal{S}$ are denoted as $h \doteq s$, if
 $h_0=s_0$, $h_{ij}=s_{ij}$, $h_{ijk}=s_{ijk}$, and $h_{ijkl}=s_{ijkl}$
 by ignoring the difference between $h_i$ and $s_i$.  Then, for
 $\lambda_0, \xi_0, \sigma^2, \tau^2 \in \mathbb{R}$, the following
additivity property holds.
\begin{equation} \label{eq:additivity}
 \mathcal{L}_{\tau^2}(
 \mathcal{L}_{\sigma^2}(h,\lambda_0),\xi_0) \doteq \mathcal{L}_{\sigma^2
 + \tau^2}(h,\lambda_0+\xi_0).
\end{equation}
As a special case, ``$\doteq$'' in (\ref{eq:additivity}) is replaced by
``$\simeq$'' if $\sigma^2 \xi_0 = \tau^2 \lambda_0$.  In particular,
the identity operator $\mathcal{L}_0(h,0) \simeq h$, and the inverse
operator
\[
 \mathcal{L}_{-\sigma^2}(
\mathcal{L}_{\sigma^2}(h,\lambda_0),-\lambda_0) \simeq h
\]
hold for the $h_i$ term too.
\end{lemm}

\subsection{Asymptotic expansion of the unbiased $p$-value}
\label{sec:asymptotic-pv}

We are now prepared to derive the expression of the fourth-order
accurate $p$-value mentioned in Section~\ref{sec:asymptotic-bp}.  We
consider a surface on which PV remains constant.  For $H=\mathcal{R}(h)$
with $h\in\mathcal{S}$, we consider a function $s(u)$ of $u\in
\mathbb{R}^q$ satisfying
 \[
  \mathrm{PV}(H|(u,-s(u))) = \alpha,\quad u\in \mathbb{R}^q,
 \]
meaning $\mathrm{PV}(H|y)=\alpha$ is constant for any
$y\in\mathcal{B}(s)$.  For a specified significance level $\alpha$, we
will reject $H_0$ if $y\not\in \mathcal{R}(s)$, and accept $H_0$ if
$y\in \mathcal{R}(s)$.  Since PV is fourth-order accurate, the
acceptance probability for any $\mu=(\theta,-h(\theta))\in \partial H$
is expressed as
\[
\mathrm{BP}(\mathcal{R}(s) | (\theta,-h(\theta))) \simeq
 1- \alpha,\quad \theta\in\mathbb{R}^q,
\]
meaning $\partial H$ is the contour surface of the bootstrap probability
of $\mathcal{R}(s)$.  

For a specified $y=(0,\lambda_0-h_0)$, we will
choose the value of $\alpha$ so that $y \in \mathcal{B}(s)$. 
Considering $(0,\lambda_0-h_0)\in \mathcal{B}(s) \Leftrightarrow
\lambda_0 -h_0 = -s_0 \Leftrightarrow (0,-\lambda_0 - s_0)\in \partial
H$, we have
\[
 h \simeq \mathcal{L}_1(s,-\lambda_0).
\]
Using the inverse operator in Lemma~\ref{lem:additivity},
 the contour surface of PV is expressed as
\[
 s \simeq \mathcal{L}_{-1}(h,\lambda_0).
\]
The expression of $\mathrm{PV}(H|y)$ will be obtained as $\alpha$
for $y \in \mathcal{B}(s)$, and thus, by choosing $\mu=(0,-h_0)$ with
$\theta=0$, we get
\[
 \mathrm{PV}(H|y) \simeq 1 - \mathrm{BP}(\mathcal{R}(s) |
 (0,-h_0)).
\]

For applying Theorem~\ref{thm:bp} to $ \mathrm{BP}(\mathcal{R}(s) |
(0,-h_0))$, we would like to replace $h\to s$ and $\lambda_0-h_0 \to
-h_0$ in $\mathrm{BP}(\mathcal{R}(h) | (0,\lambda_0-h_0))$.  This
implies replacing $\lambda_0\to-\lambda_0$ as well as
$\gamma_i\to\gamma_i(s,0)$ in (\ref{eq:bpexpgamma}), because
$\lambda_0-h_0 \to (-\lambda_0) - s_0 = -h_0$ as desired.  This is
equivalent to replacing $\beta_0\to -\beta_0$,
$\beta_1\to\beta_1-\beta_3$, $\beta_2\to\beta_2$ in (\ref{eq:bpexpbeta})
as shown in the proof of the theorem below, and therefore, we obtain
$\mathrm{PV}(H|y) \simeq 1 - \bar\Phi((-\beta_0) + (\beta_1 - \beta_3) +
\beta_2) = \bar\Phi( \beta_0 -\beta_1 - \beta_2 + \beta_3)$.

\begin{theo}[Fourth-order accurate $p$-value] \label{thm:pval}
 For the $H$ and $y=(0,\lambda_0-h_0)$ given in Theorem~\ref{thm:bp}, 
an approximately unbiased $p$-value of fourth-order accuracy is 
 expressed asymptotically as (\ref{eq:pvalexpbeta}).
\end{theo}

Related results are given in Theorem~1 of \cite{Shimodaira:2008:TRN},
from which we borrowed the idea of the inverse operator.  An unusual
asymptotic theory of ``nearly flat'' surfaces is discussed there by
utilizing Fourier transform of surfaces instead of Taylor series for
handling non-smooth surfaces such as cones.

\subsection{Asymptotic expansion of the double bootstrap probability}
\label{sec:asymptotic-dbp}

To see the robustness of DBP against computational error in the
minimization of (\ref{eq:projection}), we replace $\hat\mu(H|y)$ in
(\ref{eq:dbptausigma}) by $\tilde \mu = (\theta,-h(\theta))\in \partial
H$ for some $\theta\in\mathbb{R}^q$.  We assume $\theta=\ordz$, meaning
that the computational error is $\ordi$ with respect to the original
parameter, say, $\eta$ in the spherical example.  We denote $
\widetilde{\mathrm{DBP}}_{\tau^2,\sigma^2}(H|y) $ for this modified
double bootstrap probability, and derive its asymptotic expansion for
$y=(0,\lambda_0-h_0)$.

First note that $ \mathrm{BP}_{\sigma^2}(H|Y^{+})
\ge \mathrm{BP}_{\sigma^2}(H|y) \Leftrightarrow Y^{+}\in \mathcal{R}(s)$
for $s=\mathcal{L}_{\sigma^2}(h,\lambda_0)$, and
\[
  \widetilde{\mathrm{DBP}}_{\tau^2,\sigma^2}(H|y)  =1 -
 \mathrm{BP}_{\tau^2}(\mathcal{R}(s)| \tilde \mu).
\]
By applying Theorem~\ref{thm:bps} to $
\mathrm{BP}_{\tau^2}(\mathcal{R}(s)| \tilde \mu)$, we get the the
following theorem via a straightforward computation.

\begin{theo}[Scaling-law of the double bootstrap probability] 
\label{thm:dbpexp} For the $H$ and $y=(0,\lambda_0-h_0)$ given in
Theorem~\ref{thm:bp}, the modified double bootstrap probability with
$\tilde \mu = (\theta,-h(\theta))$ is expressed asymptotically as
\begin{equation}
 \widetilde{\mathrm{DBP}}_{\tau^2,\sigma^2}(H|y) 
 \simeq 
\bar\Phi\Bigl[ 
\beta_0 \tau^{-1} - \beta_1 \tau - \beta_2 \tau^3 - \beta_3 \tau
 \sigma^2 
-\tau^{-1}(\tau^2+\sigma^2)  \kappa(\theta)
\Bigr], \label{eq:dbpexptheta}
\end{equation} 
where $\kappa(\theta)$ is defined in (\ref{eq:kappa}).
\end{theo}

When $h_i=0$, we have $\hat\mu(H|y)=(0,-h_0)$. By letting
$\theta=0$ in (\ref{eq:dbpexptheta}), we obtain
\begin{equation} \label{eq:dbpexp}
 \mathrm{DBP}_{\tau^2,\sigma^2}(H|y) \simeq
\bar\Phi\Bigl[
\beta_0 \tau^{-1} - \beta_1 \tau - \beta_2 \tau^3 - \beta_3 \tau \sigma^2
\Bigr].
\end{equation}
When $h_i=\ordii$, we have $\hat\mu(H|y)=(\theta,-h(\theta))$ with some
$\theta=\ordii$ for which $\kappa(\theta)\simeq0$.  Therefore,
(\ref{eq:dbpexp}) holds for any $h\in\mathcal{S}$, and
(\ref{eq:dbpexps}) follows. This argument also confirms
that the four geometric quantities as well as $\beta_i$ defined at
$\theta=0$ are interpreted as those defined at $\hat\mu(H|y)$, because
$\gamma_i(h,\theta)\simeq \gamma_i$ for $\theta=\ordii$.

Comparing (\ref{eq:dbpexptheta}) with (\ref{eq:dbpexp}), we find that
$\kappa(\theta)$ represents deviation of $
\widetilde{\mathrm{DBP}}_{\tau^2,\sigma^2}(H|y)$ from
$\mathrm{DBP}_{\tau^2,\sigma^2}(H|y)$ due to computational error of
$\hat\mu(H|y)$. For $\theta=\ordz$, the deviation is
$\kappa(\theta)=\ordii$. $ \widetilde{\mathrm{DBP}}_{1,1}(H|y) =
\mathrm{DBP}_{1,1}(H|y) + \ordii$ and thus DBP is degraded from
third-order accurate to second-order accurate under the computational
error.  However, the deviation disappears in (\ref{eq:dbpexptheta}) when
$\sigma^2=-\tau^2$. In particular, $
\widetilde{\mathrm{DBP}}_{1,-1}(H|y) \simeq \mathrm{DBP}_{1,-1}(H|y)$
and thus DAU remains fourth-order accurate even if there is
computational error of $\theta=\ordz$.

Let us assume that $\partial H$ is a constant-mean-curvature surface.
Noting $\gamma_1(h,\theta)=\gamma_1$ for any $\theta=\ordz$, we have
$h_{mmi} = 0$, $6 h_{mmij} - 2 h_{mm} h_{li} h_{lj} - 4 h_{ml} h_{mi}
h_{lj} = 0$, and thus $ \kappa(\theta) = - 6\lambda_0 h_{ml} h_{mli}
\theta_i=\ordiii $. Therefore, DBP is degraded from fourth-order
accurate to third-order accurate. In addition, we may assume that
$\gamma_2(h,\theta)=\gamma_2$ for any $\theta=\ordz$, and so
$h_{ml}h_{mli}=0$; this is the case for the spherical example.  Then the
deviation $ \kappa(\theta)\simeq 0$, and DBP remains fourth-order
accurate. Therefore, DBP is as good as DAU under these conditions.

\subsection{Asymptotic accuracy of bootstrap methods}
\label{sec:asymptotic-accuracy}

For deriving the rejection probabilities (\ref{eq:rejprob-bp}) and
(\ref{eq:rejprob-dbp}) mentioned in Section~\ref{sec:intro}, here we
assume that $\mu=(0,-h_0)$ in the $(u,v)$ coordinates. Thus the
expressions of $\gamma_i$ and $\beta_i$ in Section~\ref{sec:class-s}
are now interpreted as geometric quantities defined at $\mu\in \partial
H$ instead of $\hat\mu(H|y)$.

First we consider testing $H_0$ by using $ \mathrm{NBP}_{\sigma^2}(H|y)$
as an approximate $p$-value.  For a given $\alpha$, we may choose
$\lambda_0\in\mathbb{R}$ so that
$\mathrm{NBP}_{\sigma^2}(H|(0,\lambda_0-h_0))=\alpha$.  Then the
acceptance region is expressed as $ \{ y | \mathrm{NBP}_{\sigma^2}(H|y)
\ge \alpha\} = \mathcal{R}(s) $ using
$s=\mathcal{L}_{\sigma^2}(h,\lambda_0)$, and thus
\[
P\Bigl( \mathrm{NBP}_{\sigma^2}(H|Y) < \alpha  \Bigr) =  
1-\mathrm{BP}(\mathcal{R}(s)| (0,-h_0)).
\]
This is computed as
$\widetilde{\mathrm{DBP}}_{1,\sigma^2}(H|(0,\lambda_0-h_0))$ with $\tilde
\mu=(0,-h_0)$ in the theorem below.
\begin{theo}[Rejection probability of the normalized bootstrap
 probability] \label{thm:bpacc}
 For the $H$ given in Theorem~\ref{thm:bp}, and $\mu=(0,-h_0)\in
 \partial H$, the rejection probability of
 $\mathrm{NBP}_{\sigma^2}(H|y)$ is
\begin{equation}\label{eq:rejbpexp}
\begin{split}
 P\Bigl( \mathrm{NBP}_{\sigma^2}(H|Y) < \alpha  \Bigr)
\simeq \Phi\Bigl[
z_\alpha &+(1+\sigma^2)\Bigl\{
\gamma_1+z_\alpha \gamma_2 + 
\tfrac{4}{3} z_\alpha^2 \gamma_3 - \gamma_1\gamma_2\Bigr\} \\
 & +(1+\sigma^2)^2 \Bigl\{3\gamma_4 - \tfrac{4}{3} \gamma_3
\Bigr\}-\sigma^2 \tfrac{4}{3} \gamma_3\Bigr]. 
\end{split} 
\end{equation}
In particular, $\sigma^2=1$ gives (\ref{eq:rejprob-bp}), and
$\sigma^2=-1$ gives
\[
 P\Bigl( \mathrm{AU}(H|Y) < \alpha  \Bigr) \simeq \Phi(z_\alpha +
 \tfrac{4}{3} \gamma_3) = \alpha + \ordiii.
\]
Therefore BP is first-order accurate, and AU is third-order accurate.
\end{theo}

Next we consider testing $H_0$ by using $\mathrm{DBP}_{1,\sigma^2}(H|y)$
as an approximate $p$-value.  For a given $\alpha$, we may choose
$\lambda_0\in\mathbb{R}$ so that
$\mathrm{DBP}_{1,\sigma^2}(H|(0,\lambda_0-h_0))=\alpha$.  We will see, in
the proof of the theorem below, the acceptance region is expressed as $
\{ y | \mathrm{DBP}_{1,\sigma^2}(H|y) \ge \alpha\} = \mathcal{R}(s) $
using $s=\mathcal{L}_{-1}(h,\lambda_0)$, and thus the rejection
probability is $1-\mathrm{BP}(\mathcal{R}(s)| (0,-h_0))$.  This is
computed as $\widetilde{\mathrm{DBP}}_{1,-1}(H|(0,\lambda_0-h_0))$ with
$\tilde \mu=(0,-h_0)$.

\begin{theo}[Rejection probability of the double bootstrap probability]
 \label{thm:dbpacc}
 For the $H$ given in Theorem~\ref{thm:bp}, and $\mu=(0,-h_0)\in
 \partial H$, the rejection probability of
 $\mathrm{DBP}_{1,\sigma^2}(H|y)$ is
\begin{equation}
 P\Bigl( \mathrm{DBP}_{1,\sigma^2}(H|Y) < \alpha  \Bigr)
\simeq \Phi\Bigl[
z_\alpha -(1+\sigma^2) \beta_3 \Bigr]. \label{eq:rejdbpexp} 
\end{equation}
In particular, $\sigma^2=1$ gives (\ref{eq:rejprob-dbp}), and
 $\sigma^2=-1$ gives
\[
 P\Bigl( \mathrm{DAU}(H|Y) < \alpha  \Bigr) \simeq \alpha.
\]
Therefore, DBP is third-order accurate, and DAU is fourth-order accurate.
\end{theo}

\section*{Acknowledgments}

I appreciate Toshitaka Uchiyama for helpful discussion. In particular,
Theorem~\ref{thm:bps} and Theorem~\ref{thm:pval} were first proved in
his 2005 master's thesis at Tokyo Institute of Technology via an
approach of Shimodaira (2004). I also appreciate Aizhen Ren for 
discussion about simulation study.

\appendix

\section*{Appendix}

The following lemma is used in
the proof of Theorem~\ref{thm:bp} below.

\begin{lemm}[Moments of normal random variables]
 \label{lem:normalmoments} Let $\delta_{ij}$ denote the Kronecker delta,
and indices $i,j,\ldots \in \{1,\ldots, q\}$. Consider the multivariate
normal distribution $(U_1,\ldots,U_q)\sim N_q(0,I_q)$.  Then the first
three even-order moments are
\begin{gather*}
 E(U_i U_j) = \delta_{ij}, \quad
E(U_i U_j U_k U_l) = \delta_{ij} \delta_{kl} + 
\delta_{ik} \delta_{jl} + 
\delta_{il} \delta_{jk},\\
 E(U_i U_j U_k U_l U_m U_n) =
\underbrace{ \delta_{ij} \delta_{kl} \delta_{mn} +
\delta_{ik} \delta_{jl} \delta_{mn} + \cdots
+ \delta_{in} \delta_{jk} \delta_{lm}}_\text{15 terms of partitioning
 \{i,j,k,l,m,n\} into 3 pairs}.
\end{gather*}
For $k=1,2,\ldots$, the expectation of the product of $2k$ variables
 $E(U_{i_1} \cdots U_{i_{2k}})$ is the sum of $(2k)!/(2^k k!)$ terms of
 partitioning $\{i_1,\ldots,i_{2k}\}$ into $k$ pairs, where each term is
 the product of $k$ Kronecker deltas corresponding to the $k$ pairs.  On
 the other hand, odd-order moments are all zero;
\[
 E(U_i) = E(U_i U_j U_k) = E(U_i U_j U_k U_l U_m) = \cdots =0.
\]
 
\end{lemm}
\begin{proof}[Proof of Lemma~\ref{lem:normalmoments}]
This lemma is a direct consequence of the general result of
\cite{Isserlis:1918:FPM} for $U\sim  N_q(0,\Sigma)$ with any covariance
 $\Sigma$.
\end{proof}

\begin{proof}[Proof of Theorem~\ref{thm:bp}]
We denote $Y^*=(U,V)$ in the $(u,v)$ coordinates so that
 (\ref{eq:yboot}) is expressed as
\[
 U \sim N_q(0,\sigma^2 I_q),\quad V \sim N(\lambda_0-h_0,\sigma^2). 
\]
The bootstrap probability for $\sigma^2=1$ is expressed as $P((U,V)\in
H) = P(V \le -h(U))= E\Bigl[ P(V \le -h(U)| U)\Bigr] = E\Bigl[
\Phi(-(\lambda_0-h_0) - h(U))\Bigr]$. For calculating the term in the
brackets, we consider the Taylor series
\begin{equation} \label{eq:taylorphi}
 \Phi(-a-x) =\Phi(-a) + \phi(a) \Bigl[
-x + \tfrac{1}{2}a x^2 + \tfrac{1}{6}(1-a^2) x^3 
\Bigr] + O(x^4),
\end{equation}
with $a=\lambda_0$ and $ x=h(U)-h_0
\simeq h_i U_i + h_{ij} U_i U_j +
h_{ijk} U_i U_j U_k +
h_{ijkl} U_i U_j U_k U_l.$
Then we have
\begin{equation} \label{eq:taylorbp}
 P((U,V)\in H) \simeq
\Phi(-a) + \phi(a) \Bigl[
-E(x) + \tfrac{1}{2}a E(x^2) + \tfrac{1}{6}(1-a^2) E(x^3 )
\Bigr].
\end{equation}
For calculating $E(x)$, $E(x^2)$, and $E(x^3)$, we use
Lemma~\ref{lem:normalmoments}. By noticing (\ref{eq:classscoef}), $ E(x)
\simeq h_{ij} E(U_i U_j) + h_{ijkl} E(U_i U_j U_k U_l) =h_{ij}
\delta_{ij} + h_{ijkl} (\delta_{ij} \delta_{kl} + \delta_{ik}
\delta_{jl} + \delta_{il} \delta_{jk}) =h_{ii} + 3 h_{iijj} = \gamma_1 +
3\gamma_4$.  $ E(x^2) \simeq h_{ij} h_{kl} E(U_i U_j U_k U_l) = h_{ij}
h_{kl} (\delta_{ij} \delta_{kl} + \delta_{ik} \delta_{jl} + \delta_{il}
\delta_{jk}) =h_{ii} h_{jj} + 2 h_{ij} h_{ij} = \gamma_1^2 + 2
\gamma_2$.  $ E(x^3)\simeq h_{ij} h_{kl} h_{mn} E(U_i U_j U_k U_l U_m
U_n) =h_{ij} h_{kl} h_{mn} ( \delta_{ij} \delta_{kl} \delta_{mn} +
\delta_{ik} \delta_{jl} \delta_{mn} + \cdots + \delta_{in} \delta_{jk}
\delta_{lm}) =h_{ii} h_{jj} h_{kk} + 6 h_{ii} h_{jk} h_{jk} + 8 h_{ij}
h_{jk} h_{ki} = \gamma_1^3 + 6 \gamma_1 \gamma_2 + 8 \gamma_3$.
Substituting these moments in (\ref{eq:taylorbp}), we have $P((U,V)\in
H)$ expressed as
\begin{equation} \label{eq:taylorbpe}
\Phi(-a) +\phi(a) \Bigl[
-( \gamma_1 + 3 \gamma_4)
 + \tfrac{1}{2} a
(\gamma_1^2 + 2 \gamma_2) + \tfrac{1}{6} (1-a^2)
(\gamma_1^3 + 6 \gamma_1 \gamma_2 + 8 \gamma_3)
\Bigr].
\end{equation}
Next, we consider (\ref{eq:taylorphi}) again, but with $a=\lambda_0$,
$x= \gamma_1 - a \gamma_2 + 3 \gamma_4 - \gamma_1 \gamma_2 -\tfrac{4}{3}
(1-a^2) \gamma_3$. Then we easily verify that the right hand side of
(\ref{eq:taylorphi}) gives (\ref{eq:taylorbpe}) by ignoring terms of
$\ordiiii$. Thus $P((U,V)\in H) \simeq \Phi(-a-x)=\bar \Phi(a+x)$, and
 we get (\ref{eq:bpexpgamma}).
\end{proof}

\begin{proof}[Proof of Theorem~\ref{thm:bps}]
Considering $Y^*\in H \Leftrightarrow \sigma^{-1}Y^*\in \sigma^{-1} H$,
we have $\mathrm{BP}_{\sigma^2}(H|y) = P_{\sigma^2}(Y^*\in H| y)
=P_{\sigma^2}(\sigma^{-1} Y^*\in \sigma^{-1} H| y)
=\mathrm{BP}_1(\sigma^{-1}H|\sigma^{-1}y)$, where the last equation
follows from $\sigma^{-1} Y^* \sim N_{q+1}(\sigma^{-1}y,I_{q+1})$.  This
proves (\ref{eq:rescaleregion}).  We only have to show that replacing
$H\to \sigma^{-1}H$ and $y\to \sigma^{-1}y$ in (\ref{eq:bpexpbeta})
implies (\ref{eq:rescalebeta}).  In the $(u,v)$ coordinates, the $k$-th
derivative $\partial^k h(u)/\partial u_{i1}\cdots \partial u_{ik} $ is
multiplied by $\sigma^{k-1}$, because all the terms in the numerator and
the denominator are scaled by $\sigma^{-1}$. Thus $H\to \sigma^{-1}H$ is
expressed as $h_0 \to \sigma^{-1}h_0, h_i \to h_i, h_{ij}\to \sigma
h_{ij}, h_{ijk} \to \sigma^2 h_{ijk}, h_{ijkl} \to \sigma^3 h_{ijkl}$,
and then $ \gamma_1 \to \sigma \gamma_1, \gamma_2 \to \sigma^2 \gamma_2,
\gamma_3 \to \sigma^3 \gamma_3, \gamma_4 \to \sigma^3 \gamma_4 $.  $y\to
\sigma^{-1}y$ is expressed as $ \lambda_0 \to \sigma^{-1}\lambda_0$.
Applying these rules to (\ref{eq:betabygamma}), we get
(\ref{eq:rescalebeta}).
\end{proof}

\begin{proof}[Proof of Lemma~\ref{lem:hlocal}]

A point $(\Delta u, -\tilde h(\Delta u))$ on $\partial H$ in the
$(\Delta u,\Delta v)$ coordinates is expressed as $(u+\Delta \tilde u,
-h(u+\Delta \tilde u))$ in the $(u,v)$ coordinates for some $\Delta
\tilde u = (\Delta \tilde u_1,\ldots \Delta \tilde u_q)
\in\mathbb{R}^q$.  Substituting $\Delta v = - \tilde h(\Delta u)$ in
(\ref{eq:uvlocaluv}), we have $ (u+\Delta \tilde u, -h(u+\Delta \tilde
u)) = (u,-h(u)) + \Delta u_i b_i - \tilde h(\Delta u) \|f\|^{-1}f
 $, and thus, using the definitions of $b_i$ and $f$, we get
\begin{gather}
 \Delta \tilde u_i = \Delta  u_i - \tilde h(\Delta u)
  \|f\|^{-1}
\frac{\partial h}{\partial u_i},\quad i=1,\ldots,q,
 \label{eq:solvehui}\\
 h(u+\Delta \tilde u) = h(u) + \Delta u_i \frac{\partial h}{\partial
  u_i}
+ \tilde h(\Delta u)\|f\|^{-1}.
 \label{eq:solvehv}
\end{gather}
We are going to solve these equations to find the expression of $\tilde
h(\Delta u)$ by eliminating $\Delta \tilde u$ from (\ref{eq:solvehui})
and (\ref{eq:solvehv}). We first consider the asymptotic order of the
terms in (\ref{eq:solvehui}).  $\tilde h(\Delta u)=\ordi$ because
$h(u)=\ordi$ for any $u$. $\|f\|=1+\ordii$ as shown later.  $\partial
h/\partial u_i=h_i + 2h_{ij} u_j + \cdots = \ordi$. It then follows from
(\ref{eq:solvehui}) that $\Delta \tilde u_i - \Delta u_i = O(n^{-1/2}\,
n^{-1/2} ) = \ordii$.  We next consider the Taylor expansion of $h(u +
\Delta \tilde u)$ around $u + \Delta u$.  $ h(u + \Delta \tilde u)
\simeq h(u + \Delta u) + (\partial h/\partial u_i |_{u+\Delta u})
(\Delta \tilde u_i - \Delta u_i) + O(n^{-1/2}\,\|\Delta \tilde u_i -
\Delta u_i\|^2) \simeq h(u + \Delta u) - (\partial h/\partial u_i
|_{u+\Delta u}) \tilde h(\Delta u) \|f\|^{-1} (\partial h/\partial u_i
|_u )$. Substituting this into the left hand side of
 (\ref{eq:solvehv}), we solve the equation for $\tilde h(\Delta
 u)$. Then we have $\tilde h(\Delta u) \simeq
\|f\| A B$ with
\[
A = \biggl(1 +  \frac{\partial h} {\partial u_i}\Bigr|_{u+\Delta u}
\frac{\partial h}{\partial u_i}
  \biggr)^{-1},\quad B= h(u+\Delta u) - h(u) - \Delta u_i \frac{\partial
  h}{\partial   u_i}.
\]
We look at the three factors $\|f\|$, $A$ and $B$.  The first factor is
$ \|f\| = \{ 1+ \sum_{i=1}^q (\partial h/\partial u_i)^2\}^{1/2} \simeq
1+ (1/2) \sum_{i=1}^q (\partial h/\partial u_i)^2 $. By noting $\partial
h/\partial u_i \simeq h_i + 2 h_{ij} u_j + 3h_{ijk} u_j u_k + 4 h_{ijkl}
u_j u_k u_l$, we have $ \|f\| \simeq 1 + 2 h_{ij} h_{ik} u_j u_k + 2h_i
h_{ij} u_j + 6 h_{ij} h_{ikl} u_j u_k u_l = 1 + 2 h_{ij} h_{ik} u_j u_k
+ \ordiii$.  The second factor is $A \simeq 1 - (\partial h/\partial
u_i|_{u+\Delta u}) ( \partial h/\partial u_i) \simeq 1 - \{2 h_{ij}
(u_j+\Delta u_j)+ \ordii\}\{ 2 h_{ik} u_k + \ordii\} = 1 - 4 h_{ij}
h_{ik} u_j u_k - 4 h_{ij} h_{ik} \Delta u_j u_k + \ordiii $.  The third
factor is $B \simeq (1/2) (\partial^2 h/ \partial u_i \partial
u_j)\Delta u_i \Delta u_j + (1/6)(\partial^3/\partial u_i \partial u_j
\partial u_k)\Delta u_i \Delta u_j \Delta u_k +
(1/24)(\partial^4/\partial u_i \partial u_j \partial u_k \partial
u_l)\Delta u_i \Delta u_j \Delta u_k \Delta u_l \simeq (h_{ij} + 3
h_{ijk} u_k + 6 h_{ijkl} u_k u_l) \Delta u_i \Delta u_j +(h_{ijk} + 4
h_{ijkl} u_l) \Delta u_i \Delta u_j \Delta u_k + h_{ijkl} \Delta u_i
\Delta u_j \Delta u_k \Delta u_l $.  Simply multiplying the three
factors and collect terms with respect to $\Delta u$, we
obtain $\tilde h(\Delta u) \simeq \|f\| A B
\simeq B + \Delta u_i \Delta u_j h_{ij} ( 2 h_{mk}h_{ml}u_k u_l
-4 h_{mk}h_{ml}u_k u_l) + \Delta u_i \Delta u_j \Delta u_k h_{ij}
(-4h_{mk}h_{ml}u_l)
\simeq \Delta u_i \Delta u_j \{ h_{ij} + 3
h_{ijk} u_k + (6 h_{ijkl} - 2 h_{ij} h_{mk} h_{ml}) u_k u_l\} + \Delta
u_i \Delta u_j \Delta u_k \{ h_{ijk} + (4 h_{ijkl} - 4h_{ij} h_{mk}
h_{ml}) u_l \} + \Delta u_i \Delta u_j \Delta u_k \Delta u_l h_{ijkl}
$. Looking at the coefficients, we get $\tilde h_{ij}$ and $\tilde
h_{ijkl}$. We also get $\tilde h_{ijk} = h_{ijk} + (4 h_{ijkl} -
4h_{ij} h_{mk} h_{ml}) u_l$, which becomes $\tilde h_{ijk}$ in the
lemma by symmetrization with respect to permutation of indices.

\end{proof}

\begin{proof}[Proof of Lemma~\ref{lem:glocal}]

Consider a change of coordinates $x \leftrightarrow \Delta u$ in the
tangent space as $c_i x_i = b_i \Delta u_i$. Treating $c_i$, $b_i$, $x$,
$\Delta u$ as column vectors (although they were defined as row vectors
earlier), we write $C x = B \Delta u$ in the matrix notation using
$C=(c_1,\ldots,c_q)$, $B=(b_1,\ldots, b_q)$, and thus $\Delta u = B^{-1}
C x$. Considering $\tilde h(\Delta u) = d(x)$ for any $x$, we have
$\tilde h_{ij} \Delta u_i \Delta u_j = d_{ij}x_i x_j $, $\tilde h_{ijk}
\Delta u_i \Delta u_j \Delta u_k= d_{ijk} x_i x_j x_k$, etc.
Substituting $\Delta u = B^{-1} C x$ in $\Delta u^T \tilde D \Delta u =
x^T D x$, we have $D = C^T (B^{-1})^T \tilde D B^{-1} C$, where $T$
denotes matrix transpose.  Noting $C^T C= C C^T =I$ and $B^T B = G$, we
obtain $\Tr(D) = \Tr(C^T (B^{-1})^T \tilde D B^{-1} C) = \Tr(\tilde D
B^{-1} C C^T (B^{-1})^T ) = \Tr(\tilde D B^{-1} (B^{-1})^T ) =
\Tr(\tilde D (B^T B)^{-1} ) = \Tr(\tilde D G^{-1} ) $, thus proving the
first equation for $\gamma_1(h,u)$ in (\ref{eq:uvgamma}). Similarly,
$\Tr(D^2)=\Tr((C^T (B^{-1})^T \tilde D B^{-1} C )^2 ) = \Tr((\tilde D
G^{-1} )^2 )$ for $\gamma_2(h,u)$, and $\Tr(D^3)=\Tr((C^T (B^{-1})^T
\tilde D B^{-1} C )^3 ) = \Tr((\tilde D G^{-1} )^3 )$ for
$\gamma_3(h,u)$.  For $\gamma_4(h,u)$, applying the argument of
$\gamma_1(h,u)$ twice to $(i,j)$ and $(k,l)$ in $\tilde h_{ijkl} \Delta
u_i \Delta u_j \Delta u_k \Delta u_l = d_{ijkl} x_i x_j x_k x_l$, we get
the last equation in (\ref{eq:uvgamma}).

For deriving the asymptotic expansions of $\gamma_i$'s in
 (\ref{eq:hijgamma}), we first consider $g_{ij} = b_i \cdot b_j=\delta_i
 \cdot \delta_j + (\partial h/\partial u_i) (\partial h/\partial u_j) =
 \delta_{ij} + (2 h_{ik} u_k + \ordii) (2h_{jl} u_l + \ordii) =
 \delta_{ij} + 4h_{ik} h_{jl} u_k u_l + \ordiii$. Since $(I_q+A)^{-1} =
 I_q - A + A^2 - \cdots$, the elements of $G^{-1}$ are $ g^{ij} =
 \delta_{ij} - 4h_{ik} h_{jl} u_k u_l + \ordiii$. Noting the expression
 of $\tilde h_{ij}$ shown in Lemma~\ref{lem:hlocal}, we have
 $\gamma_1(h,u)=\tilde h_{ij} g^{ij}\simeq \tilde h_{ii} -4 h_{ij}
 h_{ik} h_{jl} u_k u_l \simeq h_{ii} + 3 h_{iik} u_k + (6 h_{iikl} -2
 h_{ii} h_{mk} h_{ml} -4 h_{ij} h_{ik} h_{jl}) u_k u_l $. Also
 $\gamma_2(h,u)= \tilde h_{ij} g^{jk} \tilde h_{kl} g^{li} =\tilde
 h_{ij} (\delta_{jk}+\ordii) \tilde h_{kl} (\delta_{li}+\ordii) \simeq
 \tilde h_{ij} \tilde h_{ij} =
 (h_{ij}+3h_{ijk}u_k+\ordiii)(h_{ij}+3h_{ijl}u_l + \ordiii) \simeq
 h_{ij} h_{ij} + 6 h_{ij} h_{ijk} u_k$. Similarly, $\gamma_3(h,u)\simeq
 \tilde h_{ij} \tilde h_{jk} \tilde h_{ki} \simeq h_{ij} h_{jk} h_{ki}
 $, and $\gamma_4(h,u)\simeq \tilde h_{iijj} \simeq h_{iijj} $.

\end{proof}

\begin{proof}[Proof of Lemma~\ref{lem:twosurfaces}]
We again write $f$ for $f(u)$. By looking at each element of
 (\ref{eq:sdefinition}), we have
\begin{gather}
\theta_i = u_i + \lambda(u) \|f\|^{-1} \frac{\partial h}{\partial
 u_i},\quad
i=1,\ldots,q,
  \label{eq:sdefu}\\
 s(\theta) = h(u) - \lambda(u) \|f\|^{-1}.
 \label{eq:sdefv}
\end{gather}
We are going to solve these equations to find the expression of
$s(\theta)$ by eliminating $u$ in (\ref{eq:sdefu}) and (\ref{eq:sdefv}).
First, we rearrange the right hand side of (\ref{eq:sdefv}) to have an
expression of $a(u) = h(u) - \lambda(u) \|f\|^{-1}$.  Noting the
expression of $\|f\|$ in the proof of Lemma~\ref{lem:hlocal}, we have $
\|f\|^{-1} \simeq 1 - 2 h_{ij} h_{ik} u_j u_k - 2h_i h_{ij} u_j - 6
h_{ij} h_{ikl} u_j u_k u_l$, and then $\lambda(u) \|f\|^{-1} \simeq
\lambda(u) - \lambda_0 (1- \|f\|^{-1}) \simeq \lambda(u) - \lambda_0( 2
h_{ij} h_{ik} u_j u_k + 2h_i h_{ij} u_j + 6 h_{ij} h_{ikl} u_j u_k u_l
)$. Thus the coefficients of $a(u)$ are $a_0=h_0-\lambda_0$,
$a_i=h_i-\lambda_i + 2\lambda_0 h_m h_{mi}$, $a_{ij}=h_{ij}-\lambda_{ij}
+2\lambda_0 h_{mi} h_{mj}$, $a_{ijk} = h_{ijk} +6\lambda_0
h_{mi}h_{mjk}$, $a_{ijkl}=h_{ijkl}$.  We leave terms such as
$h_{mi}h_{mjk}=\ordiii$ in $a_{ijk}$ unsymmetrical with respect to
permutation of indices for brevity.  Next, we verify that
\begin{equation} \label{eq:solvedui}
u_i = \theta_i  - \lambda_0 h_i - 2\lambda_0 h_{ij}  \theta_j
 + 4\lambda_0^2 h_{ij} h_{jk} \theta_k
-3\lambda_0 h_{ijk} \theta_j \theta_k + \ordiii
\end{equation}
is the solution of (\ref{eq:sdefu}) up to $\ordii$ terms.  Noting
$\lambda(u) \|f\|^{-1}=\lambda_0+\ordii$ and $\partial h/\partial u_i =
h_i + 2 h_{ij} u_j + 3h_{ijk} u_j u_k + \ordiii$, (\ref{eq:sdefu}) is
expressed as $ \theta_i = u_i + \lambda_0 (h_i + 2 h_{ij} u_j + 3
h_{ijk} u_j u_k) + \ordiii$. By substituting (\ref{eq:solvedui}) into
 it, we have
 $\theta_i = u_i + \lambda_0 \{ h_i + 2h_{ij}(\theta_j -
2\lambda_0 h_{jk}\theta_k + \ordii)
+3h_{ijk}(\theta_j+\ordi)(\theta_k+\ordi)\} = u_i + \lambda_0 h_i +
2\lambda_0 h_{ij}\theta_j - 4\lambda_0^2 h_{ij} h_{jk} \theta_k +
3\lambda_0 h_{ijk}\theta_j \theta_k + \ordiii = \theta_i + \ordiii$,
confirming the solution.

We then substitute (\ref{eq:solvedui}) into $a_{i} u_i$, $a_{ij} u_i
u_j$, $a_{ijk} u_i u_j u_k$, $a_{ijkl} u_i u_j u_k u_l$. They are $a_i
u_i \simeq (a_i - 2\lambda_0 a_m h_{mi}) \theta_i $, $a_{ij} u_i u_j
\simeq (-2\lambda_0 a_{mi}h_m) \theta_i +(a_{ij} -4\lambda_0 a_{mi}
h_{mj} + 4\lambda_0^2 a_{ml} h_{mi}h_{lj} + 8\lambda_0^2 a_{mi} h_{ml}
h_{lj}) \theta_i\theta_j + (-6\lambda_0 a_{mi} h_{mjk})
\theta_i\theta_j\theta_k$, $a_{ijk} u_i u_j u_k \simeq (a_{ijk}
-6\lambda_0 a_{mij}h_{mk}) \theta_i\theta_j\theta_k$, $a_{ijkl}u_i u_j
u_k u_l \simeq a_{ijkl}\theta_i \theta_j \theta_k \theta_l$.  After
rearranging the terms of $a(u)$, we get the expression of $s(\theta)$
with respect to $\theta$ as $s(\theta) \simeq (h_0-\lambda_0) +
(h_i-\lambda_i - 2\lambda_0 h_{mi} (h_m - \lambda_m)) \theta_i +
(h_{ij}-\lambda_{ij} - 2\lambda_0 h_{mi} h_{mj} + 4\lambda_0^2 h_{ml}
h_{mi} h_{lj}) \theta_i \theta_j + (h_{ijk} - 6\lambda_0 h_{mi} h_{mjk})
\theta_i \theta_j \theta_k + h_{ijkl} \theta_i \theta_j \theta_k
\theta_l$. Therefore, we obtain the coefficients $s_0$, $s_i$, $s_{ij}$
and $s_{ijkl}$ as those given in the lemma. We also get $s_{ijk} =
h_{ijk} - 6\lambda_0 h_{mi} h_{mjk}$, which becomes that given in the
lemma by symmetrization with respect to permutation of
indices. 
\end{proof}

\begin{proof}[Proof of Lemma~\ref{lem:contourbp}]
For $(\theta,-s(\theta))$ in (\ref{eq:sdefinition}), we will solve $
\mathrm{BP}_{\sigma^2}( \mathcal{R}(h) | (\theta,-s(\theta))) = 1
-\alpha $ with respect to $\lambda(u)$.  We apply Theorem~\ref{thm:bps}
to $y=(\theta,-s(\theta))$.  Let $\hat\gamma_i=\gamma_i(h,u)$,
$i=1,\ldots,4$ be the geometric quantities at $(u,-h(u)) =
\hat\mu(H|(\theta, -s(\theta)))$.  It follows from (\ref{eq:bpsexp}) by
replacing $\lambda_0\to\lambda(u)$, $\gamma_i\to \hat\gamma_i$,
$h_0=h_i=0$ that
\[
z_\alpha \simeq \sigma^{-1} \lambda(u) + \sigma \Bigl(\hat\gamma_1 - \lambda(u)
 \hat\gamma_2 + \tfrac{4}{3} \lambda(u)^2 \hat\gamma_3\Bigr)
+ \sigma^3 \Bigl(3\hat\gamma_4 - \hat\gamma_1 \hat\gamma_2 - \tfrac{4}{3}
 \hat\gamma_3 \Bigr).
\]
Solving this equation with respect to $\lambda(u)$, we obtain
\begin{equation} \label{eq:lambdauassolution}
 \lambda(u)\simeq
\sigma z_\alpha - \sigma^2 \hat\gamma_1 + \sigma^3 z_\alpha \hat\gamma_2
+ \sigma^4 \Bigl( \tfrac{4}{3} (1-z_\alpha^2) \hat\gamma_3- 
3\hat\gamma_4\Bigr),
\end{equation}
which is easily verified by substituting (\ref{eq:lambdauassolution})
into the equation as $z_\alpha \simeq \sigma^{-1} \lambda(u) + \sigma \{
\hat\gamma_1 - (\sigma z_\alpha - \sigma^2 \hat \gamma_1 + \ordii) \hat
\gamma_2 + \tfrac{4}{3}(\sigma z_\alpha + \ordi)^2 \hat\gamma_3 \} +
\sigma^3 (3 \hat\gamma_4 - \hat\gamma_1 \hat\gamma_2 - \tfrac{4}{3} \hat
\gamma_3) \simeq \sigma^{-1} \lambda(u) + \sigma \hat\gamma_1 + \sigma^2
(-z_\alpha \hat \gamma_2) + \sigma^3 (\tfrac{4}{3}(z_\alpha^2 - 1)
\hat\gamma_3 + 3 \hat\gamma_4) \simeq z_\alpha$.  Substituting
(\ref{eq:hijgamma}) for $\hat\gamma_i=\gamma_i(h,u)$ in
(\ref{eq:lambdauassolution}), we have $\lambda_0=\lambda(0)$,
$\lambda_i=-3\sigma^2 h_{mmi}+ 6 \sigma^3 z_\alpha h_{ml} h_{mli}$ and
$\lambda_{ij} = \sigma^2 (-6 h_{mmij} +2 h_{mm} h_{li} h_{lj} + 4 h_{ml}
h_{mi} h_{lj})$.  For proving (\ref{eq:lambdaucontour}), we eliminate
$\alpha$ from $\lambda_i$ by $z_\alpha = \sigma^{-1} \lambda_0 + \ordi$,
and we get $\lambda_i=\sigma^2 ( -3 h_{mmi} + 6\lambda_0 h_{ml}
h_{mli})$.  By applying Theorem~\ref{lem:twosurfaces} to this
$\lambda(u)$, we obtain (\ref{eq:contourcoef}) and
(\ref{eq:contourgamma}); actually we only have to check $s_i$, $s_{ij}$,
and $\gamma_1(s,0)$ as follows.  $s_i=h_i - \sigma^2(-3h_{mmi} +
6\lambda_0 h_{ml} h_{mli}) -2 \lambda_0 h_{mi} (h_m -
\sigma^2(-3h_{mll}+\ordiii))$, $s_{ij}=h_{ij}-\sigma^2 (-6 h_{mmij} +2
h_{mm} h_{li} h_{lj} + 4 h_{ml} h_{mi} h_{lj}) - 2\lambda_0 h_{mi}
h_{mj} + 4\lambda_0^2 h_{ml} h_{mi} h_{lj}$, and $\gamma_1(s,0) =
\gamma_1 - \lambda_{ii} - 2\lambda_0 \gamma_2 + 4\lambda_0^2 \gamma_3$
with $\lambda_{ii} = \sigma^2 (-6\gamma_4 + 2\gamma_1 \gamma_2 +
4\gamma_3)=-\sigma^2 \beta_3$.
\end{proof}

\begin{proof}[Proof of Lemma~\ref{lem:additivity}]
Let $s = \mathcal{L}_{\sigma^2}(h,\lambda_0)$ and $r =
\mathcal{L}_{\tau^2}(s,\xi_0)$. Applying Lemma~\ref{lem:contourbp} to
$r$, we have the coefficients of $r$ in terms of $s$ and $\xi_0$ such as
$r_0=s_0-\xi_0$ from (\ref{eq:contourcoef}).  Then substitute the
coefficients of $s$ in terms of $h$ and $\lambda_0$, such as
$s_0=h_0-\lambda_0$, into those of $r$ to get, say,
$r_0=(h_0-\lambda_0)-\xi_0 = h_0 - (\lambda_0+\xi_0)$.  After
rearranging terms, we get the other coefficients as $r_{ij}=h_{ij} -
2(\lambda_0+\xi_0)h_{mi}h_{mj} + 4(\lambda_0+\xi_0)^2 h_{ml} h_{mi}
h_{lj} + (\sigma^2 + \tau^2) (6 h_{mmij} - 2 h_{mm} h_{li} h_{lj} - 4
h_{ml} h_{mi} h_{lj})$, $r_{ijk} = h_{ijk}- 2(\lambda_0+\xi_0) (h_{mi}
h_{mjk}+h_{mj} h_{mik}+h_{mk} h_{mij} ) $, $r_{ijkl}=h_{ijkl}$. The
additivity in terms of $\sigma^2+\tau^2$ and $\lambda_0+\xi_0$ holds for
these four coefficients, thus proving (\ref{eq:additivity}). For $r_i =
h_i - 2(\lambda_0 +\xi_0) h_m h_{mi} + (\sigma^2 + \tau^2) (3h_{mmi} -
6(\lambda_0+\xi_0) h_{ml} h_{mli} -6(\lambda_0 + \xi_0) h_{mi} h_{mll})
+ (\sigma^2 \xi_0 - \tau^2 \lambda_0)6h_{ml} h_{mli}$, the additivity
holds except for the last term. Thus ``$\doteq$'' in
(\ref{eq:additivity}) is replaced by ``$\simeq$'' when $\sigma^2 \xi_0 -
\tau^2 \lambda_0=0$. In particular,
$\mathcal{L}_{-\sigma^2}(\mathcal{L}_{\sigma^2}(h,\lambda_0),-\lambda_0)
\simeq \mathcal{L}_0(h,0) \simeq h$.
\end{proof}

\begin{proof}[Proof of Theorem~\ref{thm:pval}]
We only have to show that $\gamma_i\to\gamma_i(s,0)$,
$\lambda_0\to-\lambda_0$ with $s=\mathcal{L}_{-1}(h,\lambda_0)$ leads to
$\beta_0\to -\beta_0$, $\beta_1\to\beta_1-\beta_3$, $\beta_2\to\beta_2$
as mentioned just before the theorem.  Here we show a generalized result
for $s=\mathcal{L}_{\sigma^2}(h,\lambda_0)$ to be used later again.  The
geometric quantities are given in (\ref{eq:contourgamma}) of
Lemma~\ref{lem:contourbp}.  We replace $\gamma_i\to\gamma_i(s,0)$ and
$\lambda_0\to-\lambda_0$ in $\beta_0,\beta_1,\beta_2$ of
(\ref{eq:betabygamma}).  They become $\beta_0 \to -\lambda_0 =
-\beta_0$, $\beta_1 \to (\gamma_1 -2\lambda_0 \gamma_2 + 4\lambda_0^2
\gamma_3) +\sigma^2 \beta_3 - (-\lambda_0)(\gamma_2 - 4\lambda_0
\gamma_3) + \tfrac{4}{3} (-\lambda_0)^2 \gamma_3 = \gamma_1 - \lambda_0
\gamma_2 + \tfrac{4}{3} \lambda_0^2 \gamma_3 + \sigma^2 \beta_3 =
\beta_1 + \sigma^2\beta_3$, $\beta_2 \to 3\gamma_4 - (\gamma_1 +
\ordii)(\gamma_2 + \ordiii) - \tfrac{4}{3}\gamma_3 \simeq \beta_2$.
\end{proof}

\begin{proof}[Proof of Theorem~\ref{thm:dbpexp}]
We first consider the case of $\tilde \mu=(0,-h_0)$ with $\theta=0$.
For applying Theorem~\ref{thm:bps} to $
\mathrm{BP}_{\tau^2}(\mathcal{R}(s)| (0, -h_0)$ with
$s=\mathcal{L}_{\sigma^2}(h,\lambda_0)$, we would like to replace $h\to
s$ and $\lambda_0-h_0 \to -h_0$ in $\mathrm{BP}_{\tau^2}(\mathcal{R}(h)
| (0,\lambda_0-h_0))$.  We replace $\sigma^2\to\tau^2$,
$\gamma_i\to\gamma_i(s,0)$, $\lambda_0\to-\lambda_0$ in
(\ref{eq:bpsexp}).  This results in $\beta_0\to-\beta_0$,
$\beta_1\to\beta_1 + \sigma^2 \beta_3$, $\beta_2\to\beta_2$ as shown in
the proof of Theorem~\ref{thm:pval}, and thus
$\widetilde{\mathrm{DBP}}_{\tau^2,\sigma^2}(H|y)\simeq
1-\bar\Phi((-\beta_0)\tau^{-1} + (\beta_1+\sigma^2 \beta_3)\tau +
\beta_2 \tau^3 )$, giving the right hand side of (\ref{eq:dbpexp}).

Next, we compute $\widetilde{\mathrm{DBP}}_{\tau^2,\sigma^2}(H|y) $ with
$\tilde \mu=(\theta,-h(\theta))$ for $\theta=\ordz$. We only have to
replace $\beta_0,\ldots,\beta_3$ in (\ref{eq:dbpexp}) by those evaluated
at $\tilde\mu$, denoted $\tilde\beta_0,\ldots,\tilde\beta_3$.  Replacing
$\lambda_0\to \lambda_{\sigma^2}(\theta)$,
$\gamma_i\to\gamma_i(h,\theta)$ in (\ref{eq:betabygamma}), we have
$\tilde\beta_0 = \lambda_{\sigma^2}(\theta) = \beta_0 - \sigma^2
\kappa(\theta)$, $\tilde\beta_1 \simeq \gamma_1(h,\theta) -
\lambda_{\sigma^2}(\theta) \gamma_2(h,\theta) +
\tfrac{4}{3}\lambda_{\sigma^2}(\theta)^2 \gamma_3(h,\theta) = \beta_1 +
\kappa(\theta)$, $\tilde \beta_2\simeq \beta_2$, $\tilde \beta_3 \simeq
\beta_3$. Therefore,
$\widetilde{\mathrm{DBP}}_{\tau^2,\sigma^2}(H|y)\simeq \bar\Phi(
(\beta_0 - \sigma^2 \kappa(\theta))\tau^{-1} - (\beta_1 +
\kappa(\theta)) \tau - \beta_2 \tau^3 - \beta_3 \tau \sigma^2))$, giving
(\ref{eq:dbpexptheta}).
\end{proof}

\begin{proof}[Proof of Theorem~\ref{thm:bpacc}]
Let $\lambda_0\in\mathbb{R}$ be the solution of the equation
$\mathrm{NBP}_{\sigma^2}(H|(0,\lambda_0-h_0))=\alpha$. From
(\ref{eq:nbps}), the equation is expressed as $\beta_0 + \beta_1
\sigma^2 + \beta_2 \sigma^4 \simeq -z_\alpha$ with
(\ref{eq:betabygamma}).  By solving it with respect to $\lambda_0$, we
get $ \lambda_0 = -z_\alpha - \sigma^2(\gamma_1 + z_\alpha \gamma_2 +
\tfrac{4}{3} z_\alpha^2 \gamma_3) - \sigma^4 (3\gamma_4 - \tfrac{4}{3}
\gamma_3)$. This is easily verified by substituting it into the left
hand side of the equation as $\lambda_0 + \sigma^2( \gamma_1 -
(-z_\alpha - \sigma^2 \gamma_1) \gamma_2 + \tfrac{4}{3}(-z_\alpha)^2
\gamma_3) + \sigma^4 \beta_2 \simeq -z_\alpha$.  For $\tilde
\mu=(0,-h_0)$, the right hand side of (\ref{eq:dbpexp}) gives
$\widetilde{\mathrm{DBP}}_{1,\sigma^2}(H|(0,\lambda_0-h_0)) \simeq
\Phi(-\beta_0 + \beta_1 + \beta_2 + \sigma^2\beta_3) \simeq
\Phi(-\lambda_0 + \gamma_1 - (-z_\alpha - \sigma^2 \gamma_1) \gamma_2 +
 \tfrac{4}{3} (-z_\alpha)^2 \gamma_3 + \beta_2 + \sigma^2\beta_3
 )$. This becomes (\ref{eq:rejbpexp}) by collecting terms with respect
 to $\sigma^2$ after substituting the expression of $\lambda_0$.
\end{proof}

\begin{proof}[Proof of Theorem~\ref{thm:dbpacc}]
Let $\lambda_0\in\mathbb{R}$ be the solution of the equation
$\mathrm{DBP}_{1,\sigma^2}(H|(0,\lambda_0-h_0))=\alpha$.  From
(\ref{eq:dbpexp}), the equation is expressed as $\beta_0 - \beta_1 -
\beta_2 - \sigma^2 \beta_3 \simeq -z_\alpha$ with
(\ref{eq:betabygamma}).  By solving it with respect to $\lambda_0$, we
get $ \lambda_0 = -z_\alpha +\gamma_1 + z_\alpha \gamma_2 - 2\gamma_1
\gamma_2 + 3\gamma_4 + \tfrac{4}{3} \gamma_3 (z_\alpha^2-1) + \sigma^2
\beta_3 $. We define $\lambda \in\mathcal{S}$ by substituting
$\gamma_i(h,u)$ for $\gamma_i$ in the expression of $\lambda_0$. Then
$\lambda(u) \simeq -z_\alpha +\gamma_1(h,u) + z_\alpha \gamma_2(h,u) -
2\gamma_1 \gamma_2 + 3\gamma_4 + \tfrac{4}{3} \gamma_3 (z_\alpha^2-1) +
\sigma^2 \beta_3 \simeq \lambda_0 + (3h_{mmi} + 6 z_\alpha h_{ml}
h_{mli}) u_i + ( 6h_{mmij} - 2\gamma_1 h_{mi} h_{mj} - 4h_{ml} h_{mi}
h_{lj}) u_i u_j$. Noting $z_\alpha = -\lambda_0 + \ordi$, we find
$\lambda(u) \simeq \lambda_0 + \kappa(u)$, where $\kappa(u)$ is defined
in (\ref{eq:kappa}). Using this $\lambda(u)$, the contour surface
$\mathrm{DBP}_{1,\sigma^2}(H|y)=\alpha$ is expressed as
$y\in\mathcal{B}(s)$ for $s=\mathcal{M}(h,\lambda) \simeq
\mathcal{L}_{-1}(h,\lambda_0)$.  For $\tilde \mu=(0,-h_0)$, the right
hand side of (\ref{eq:dbpexp}) gives
$\widetilde{\mathrm{DBP}}_{1,-1}(H|(0,\lambda_0-h_0)) \simeq
\Phi(-\beta_0 + \beta_1 + \beta_2 - \beta_3) \simeq \Phi(-\lambda_0 +
\gamma_1 - (-z_\alpha + \gamma_1 )\gamma_2 + \tfrac{4}{3} (-z_\alpha)^2
\gamma_3 + \beta_2 -\beta_3) \simeq \Phi(z_\alpha -(1+\sigma^2)
 \beta_3)$, showing (\ref{eq:rejdbpexp}).
\end{proof}


\bibliographystyle{imsart-nameyear}
\bibliography{stat2013}

\begin{thebibliography}{32}

\bibitem[\protect\citeauthoryear{Beran}{1987}]{Beran:1987:PRL}
\begin{barticle}[author]
\bauthor{\bsnm{Beran},~\bfnm{Rudolf}\binits{R.}}
(\byear{1987}).
\btitle{Prepivoting to Reduce Level Error of Confidence Sets}.
\bjournal{Biometrika}
\bvolume{74}
\bpages{457-468}.
\end{barticle}
\endbibitem

\bibitem[\protect\citeauthoryear{Bickel, G\"{o}tze and van
  Zwet}{1997}]{Bickel:Gotze:vanZwet:1997:RFT}
\begin{barticle}[author]
\bauthor{\bsnm{Bickel},~\bfnm{P.}\binits{P.}},
  \bauthor{\bsnm{G\"{o}tze},~\bfnm{F.}\binits{F.}} \AND
  \bauthor{\bparticle{van} \bsnm{Zwet},~\bfnm{W.}\binits{W.}}
(\byear{1997}).
\btitle{Resampling fewer than $n$ observations: Gains, losses and remedies for
  losses}.
\bjournal{Statistica Sinica}
\bvolume{7}
\bpages{1--31}.
\end{barticle}
\endbibitem

\bibitem[\protect\citeauthoryear{Cook and
  Stefanski}{1994}]{Cook:Stefanski:1994:SIE}
\begin{barticle}[author]
\bauthor{\bsnm{Cook},~\bfnm{J.~R.}\binits{J.~R.}} \AND
  \bauthor{\bsnm{Stefanski},~\bfnm{L.~A.}\binits{L.~A.}}
(\byear{1994}).
\btitle{Simulation-Extrapolation Estimation in Parametric Measurement Error
  Models}.
\bjournal{Journal of the American Statistical Association}
\bvolume{89}
\bpages{1314-1328}.
\end{barticle}
\endbibitem

\bibitem[\protect\citeauthoryear{DiCiccio and
  Efron}{1992}]{Diciccio:Efron:1992:MAC}
\begin{barticle}[author]
\bauthor{\bsnm{DiCiccio},~\bfnm{Thomas}\binits{T.}} \AND
  \bauthor{\bsnm{Efron},~\bfnm{Bradley}\binits{B.}}
(\byear{1992}).
\btitle{More accurate confidence intervals in exponential families}.
\bjournal{Biometrika}
\bvolume{79}
\bpages{231--245}.
\end{barticle}
\endbibitem

\bibitem[\protect\citeauthoryear{DiCiccio and
  Efron}{1996}]{Diciccio:Efron:1996:BCI}
\begin{barticle}[author]
\bauthor{\bsnm{DiCiccio},~\bfnm{Thomas~J.}\binits{T.~J.}} \AND
  \bauthor{\bsnm{Efron},~\bfnm{Bradley}\binits{B.}}
(\byear{1996}).
\btitle{Bootstrap Confidence Intervals}.
\bjournal{Statistical Science}
\bvolume{11}
\bpages{189-212}.
\end{barticle}
\endbibitem

\bibitem[\protect\citeauthoryear{Efron}{1985}]{Efron:1985:BCI}
\begin{barticle}[author]
\bauthor{\bsnm{Efron},~\bfnm{Bradley}\binits{B.}}
(\byear{1985}).
\btitle{{Bootstrap} Confidence Intervals for a Class of Parametric Problems}.
\bjournal{Biometrika}
\bvolume{72}
\bpages{45--58}.
\end{barticle}
\endbibitem

\bibitem[\protect\citeauthoryear{Efron}{1987}]{Efron:1987:BBC}
\begin{barticle}[author]
\bauthor{\bsnm{Efron},~\bfnm{Bradley}\binits{B.}}
(\byear{1987}).
\btitle{{Better} Bootstrap Confidence Intervals}.
\bjournal{Journal of the American Statistical Association}
\bvolume{82}
\bpages{171--185}.
\end{barticle}
\endbibitem

\bibitem[\protect\citeauthoryear{Efron, Halloran and
  Holmes}{1996}]{Efron:Halloran:Holmes:1996:BCL}
\begin{barticle}[author]
\bauthor{\bsnm{Efron},~\bfnm{Bradley}\binits{B.}},
  \bauthor{\bsnm{Halloran},~\bfnm{Elizabeth}\binits{E.}} \AND
  \bauthor{\bsnm{Holmes},~\bfnm{Susan}\binits{S.}}
(\byear{1996}).
\btitle{Bootstrap confidence levels for phylogenetic trees}.
\bjournal{Proc. Natl. Acad. Sci. USA}
\bvolume{93}
\bpages{13429-13434}.
\end{barticle}
\endbibitem

\bibitem[\protect\citeauthoryear{Efron and
  Tibshirani}{1993}]{Efron:Tibshirani:1993:ITB}
\begin{bbook}[author]
\bauthor{\bsnm{Efron},~\bfnm{Bradley}\binits{B.}} \AND
  \bauthor{\bsnm{Tibshirani},~\bfnm{Robert~J.}\binits{R.~J.}}
(\byear{1993}).
\btitle{An Introduction to the Bootstrap}.
\bpublisher{Chapman \& Hall}, \baddress{New York}.
\end{bbook}
\endbibitem

\bibitem[\protect\citeauthoryear{Efron and
  Tibshirani}{1998}]{Efron:Tibshirani:1998:PR}
\begin{barticle}[author]
\bauthor{\bsnm{Efron},~\bfnm{B.}\binits{B.}} \AND
  \bauthor{\bsnm{Tibshirani},~\bfnm{R.}\binits{R.}}
(\byear{1998}).
\btitle{The problem of regions}.
\bjournal{Annals of Statistics}
\bvolume{26}
\bpages{1687--1718}.
\end{barticle}
\endbibitem

\bibitem[\protect\citeauthoryear{Felsenstein}{1985}]{Felsenstein:1985:CLP}
\begin{barticle}[author]
\bauthor{\bsnm{Felsenstein},~\bfnm{Joseph}\binits{J.}}
(\byear{1985}).
\btitle{Confidence limits on phylogenies: an approach using the bootstrap}.
\bjournal{Evolution}
\bvolume{39}
\bpages{783-791}.
\end{barticle}
\endbibitem

\bibitem[\protect\citeauthoryear{Felsenstein and
  Kishino}{1993}]{Felsenstein:Kishino:1993:ITS}
\begin{barticle}[author]
\bauthor{\bsnm{Felsenstein},~\bfnm{J.}\binits{J.}} \AND
  \bauthor{\bsnm{Kishino},~\bfnm{H.}\binits{H.}}
(\byear{1993}).
\btitle{Is there something wrong with the bootstrap on phylogenies? {A} reply
  to {Hillis} and {Bull}}.
\bjournal{Systematic Biology}
\bvolume{42}
\bpages{193--200}.
\end{barticle}
\endbibitem

\bibitem[\protect\citeauthoryear{Hall}{1986}]{Hall:1986:BCI}
\begin{barticle}[author]
\bauthor{\bsnm{Hall},~\bfnm{Peter}\binits{P.}}
(\byear{1986}).
\btitle{On the Bootstrap and Confidence Intervals}.
\bjournal{Annals of Statistics}
\bvolume{14}
\bpages{1431-1452}.
\end{barticle}
\endbibitem

\bibitem[\protect\citeauthoryear{Hall}{1992}]{Hall:1992:BEE}
\begin{bbook}[author]
\bauthor{\bsnm{Hall},~\bfnm{Peter}\binits{P.}}
(\byear{1992}).
\btitle{The bootstrap and {E}dgeworth expansion}.
\bpublisher{Springer-Verlag}, \baddress{New York}.
\end{bbook}
\endbibitem

\bibitem[\protect\citeauthoryear{Hall and
  Maesono}{2000}]{Hall:Maesono:2000:WBA}
\begin{barticle}[author]
\bauthor{\bsnm{Hall},~\bfnm{Peter}\binits{P.}} \AND
  \bauthor{\bsnm{Maesono},~\bfnm{Yoshihiko}\binits{Y.}}
(\byear{2000}).
\btitle{A Weighted Bootstrap Approach to Bootstrap Iteration}.
\bjournal{Journal of the Royal Statistical Society Series B}
\bvolume{62}
\bpages{137-144}.
\end{barticle}
\endbibitem

\bibitem[\protect\citeauthoryear{Hillis and Bull}{1993}]{Hillis:Bull:1993:ETB}
\begin{barticle}[author]
\bauthor{\bsnm{Hillis},~\bfnm{D.~M.}\binits{D.~M.}} \AND
  \bauthor{\bsnm{Bull},~\bfnm{J.~J.}\binits{J.~J.}}
(\byear{1993}).
\btitle{An empirical test of bootstrapping as a method for assessing confidence
  in phylogenetic analysis}.
\bjournal{Systematic Biology}
\bvolume{42}
\bpages{182--192}.
\end{barticle}
\endbibitem

\bibitem[\protect\citeauthoryear{Hinkley and Shi}{1989}]{Hinkley:Shi:1989:ISN}
\begin{barticle}[author]
\bauthor{\bsnm{Hinkley},~\bfnm{D.~V.}\binits{D.~V.}} \AND
  \bauthor{\bsnm{Shi},~\bfnm{S.}\binits{S.}}
(\byear{1989}).
\btitle{Importance Sampling and the Nested Bootstrap}.
\bjournal{Biometrika}
\bvolume{76}
\bpages{435-446}.
\end{barticle}
\endbibitem

\bibitem[\protect\citeauthoryear{Hsu}{1981}]{Hsu:1981:SCI}
\begin{barticle}[author]
\bauthor{\bsnm{Hsu},~\bfnm{Jason~C.}\binits{J.~C.}}
(\byear{1981}).
\btitle{Simultaneous Confidence Intervals for all Distances from the ``Best''}.
\bjournal{Annals of Statistics}
\bvolume{9}
\bpages{1026-1034}.
\end{barticle}
\endbibitem

\bibitem[\protect\citeauthoryear{Isserlis}{1918}]{Isserlis:1918:FPM}
\begin{barticle}[author]
\bauthor{\bsnm{Isserlis},~\bfnm{L.}\binits{L.}}
(\byear{1918}).
\btitle{On a Formula for the Product-Moment Coefficient of any Order of a
  Normal Frequency Distribution in any Number of Variables}.
\bjournal{Biometrika}
\bvolume{12}
\bpages{134--139}.
\end{barticle}
\endbibitem

\bibitem[\protect\citeauthoryear{Lee and Young}{1995}]{Lee:Young:1995:AIB}
\begin{barticle}[author]
\bauthor{\bsnm{Lee},~\bfnm{Stephen M.~S.}\binits{S.~M.~S.}} \AND
  \bauthor{\bsnm{Young},~\bfnm{G.~Alastair}\binits{G.~A.}}
(\byear{1995}).
\btitle{Asymptotic Iterated Bootstrap Confidence Intervals}.
\bjournal{Annals of Statistics}
\bvolume{23}
\bpages{1301-1330}.
\end{barticle}
\endbibitem

\bibitem[\protect\citeauthoryear{Liu and Singh}{1997}]{Liu:Singh:1997:NLP}
\begin{barticle}[author]
\bauthor{\bsnm{Liu},~\bfnm{Regina~Y.}\binits{R.~Y.}} \AND
  \bauthor{\bsnm{Singh},~\bfnm{Kesar}\binits{K.}}
(\byear{1997}).
\btitle{Notions of limiting {$P$} values based on data depth and bootstrap}.
\bjournal{J. Amer. Statist. Assoc.}
\bvolume{92}
\bpages{266--277}.
\end{barticle}
\endbibitem

\bibitem[\protect\citeauthoryear{Loh}{1987}]{Loh:1987:CCC}
\begin{barticle}[author]
\bauthor{\bsnm{Loh},~\bfnm{Wei-Yin}\binits{W.-Y.}}
(\byear{1987}).
\btitle{Calibrating Confidence Coefficients}.
\bjournal{Journal of the American Statistical Association}
\bvolume{82}
\bpages{155-162}.
\end{barticle}
\endbibitem

\bibitem[\protect\citeauthoryear{Martin}{1990}]{Martin:1990:BIC}
\begin{barticle}[author]
\bauthor{\bsnm{Martin},~\bfnm{Michael~A.}\binits{M.~A.}}
(\byear{1990}).
\btitle{On Bootstrap Iteration for Coverage Correction in Confidence
  Intervals}.
\bjournal{Journal of the American Statistical Association}
\bvolume{85}
\bpages{1105-1118}.
\end{barticle}
\endbibitem

\bibitem[\protect\citeauthoryear{McCullagh}{1984}]{Mccullagh:1984:LS}
\begin{barticle}[author]
\bauthor{\bsnm{McCullagh},~\bfnm{Peter}\binits{P.}}
(\byear{1984}).
\btitle{Local Sufficiency}.
\bjournal{Biometrika}
\bvolume{71}
\bpages{233-244}.
\end{barticle}
\endbibitem

\bibitem[\protect\citeauthoryear{Newton}{1996}]{Newton:1996:BPL}
\begin{barticle}[author]
\bauthor{\bsnm{Newton},~\bfnm{M.~A.}\binits{M.~A.}}
(\byear{1996}).
\btitle{Bootstrapping phylogenies: Large deviations and dispersion effects}.
\bjournal{Biometrika}
\bvolume{83}
\bpages{315-328}.
\end{barticle}
\endbibitem

\bibitem[\protect\citeauthoryear{Newton and
  Geyer}{1994}]{Newton:Geyer:1994:BRM}
\begin{barticle}[author]
\bauthor{\bsnm{Newton},~\bfnm{Michael~A.}\binits{M.~A.}} \AND
  \bauthor{\bsnm{Geyer},~\bfnm{Charles~J.}\binits{C.~J.}}
(\byear{1994}).
\btitle{Bootstrap Recycling: A Monte Carlo Alternative to the Nested
  Bootstrap}.
\bjournal{Journal of the American Statistical Association}
\bvolume{89}
\bpages{905-912}.
\end{barticle}
\endbibitem

\bibitem[\protect\citeauthoryear{Politis and
  Romano}{1994}]{Politis:Romano:1994:LSC}
\begin{barticle}[author]
\bauthor{\bsnm{Politis},~\bfnm{D.}\binits{D.}} \AND
  \bauthor{\bsnm{Romano},~\bfnm{J.}\binits{J.}}
(\byear{1994}).
\btitle{Large sample confidence regions on subsamples under minimal
  assumptions}.
\bjournal{Annals of Statistics}
\bvolume{22}
\bpages{2031-2050}.
\end{barticle}
\endbibitem

\bibitem[\protect\citeauthoryear{Shimodaira}{2002}]{Shimodaira:2002:AUT}
\begin{barticle}[author]
\bauthor{\bsnm{Shimodaira},~\bfnm{Hidetoshi}\binits{H.}}
(\byear{2002}).
\btitle{An Approximately Unbiased Test of Phylogenetic Tree Selection}.
\bjournal{Systematic Biology}
\bvolume{51}
\bpages{492--508}.
\end{barticle}
\endbibitem

\bibitem[\protect\citeauthoryear{Shimodaira}{2004}]{Shimodaira:2004:AUT}
\begin{barticle}[author]
\bauthor{\bsnm{Shimodaira},~\bfnm{Hidetoshi}\binits{H.}}
(\byear{2004}).
\btitle{Approximately unbiased tests of regions using multistep-multiscale
  bootstrap resampling}.
\bjournal{Annals of Statistics}
\bvolume{32}
\bpages{2616-2641}.
\end{barticle}
\endbibitem

\bibitem[\protect\citeauthoryear{Shimodaira}{2008}]{Shimodaira:2008:TRN}
\begin{barticle}[author]
\bauthor{\bsnm{Shimodaira},~\bfnm{Hidetoshi}\binits{H.}}
(\byear{2008}).
\btitle{Testing Regions with Nonsmooth Boundaries via Multiscale Bootstrap}.
\bjournal{Journal of Statistical Planning and Inference}
\bvolume{138}
\bpages{1227-1241}.
\end{barticle}
\endbibitem

\bibitem[\protect\citeauthoryear{Shimodaira and
  Hasegawa}{2001}]{Shimodaira:Hasegawa:2001:CAC}
\begin{barticle}[author]
\bauthor{\bsnm{Shimodaira},~\bfnm{Hidetoshi}\binits{H.}} \AND
  \bauthor{\bsnm{Hasegawa},~\bfnm{Masami}\binits{M.}}
(\byear{2001}).
\btitle{{CONSEL:}~for assessing the confidence of phylogenetic tree selection}.
\bjournal{Bioinformatics}
\bvolume{17}
\bpages{1246--1247}.
\end{barticle}
\endbibitem

\bibitem[\protect\citeauthoryear{Suzuki and
  Shimodaira}{2006}]{Suzuki:Shimodaira:2006:PRA}
\begin{barticle}[author]
\bauthor{\bsnm{Suzuki},~\bfnm{Ryota}\binits{R.}} \AND
  \bauthor{\bsnm{Shimodaira},~\bfnm{Hidetoshi}\binits{H.}}
(\byear{2006}).
\btitle{Pvclust: an {R} package for assessing the uncertainty in hierarchical
  clustering}.
\bjournal{Bioinformatics}
\bvolume{22}
\bpages{1540-1542}.
\end{barticle}
\endbibitem

\end{thebibliography}

\end{document}